\documentclass[12pt,reqno]{amsart}
\usepackage{etex}
\usepackage{dsfont,amsmath,amsfonts,amscd,amssymb,graphicx,mathrsfs,eufrak,upgreek}
\usepackage[dvips,all,arc,curve,color,frame]{xy}
\usepackage[dvipsnames]{xcolor}
\usepackage{setspace}
\usepackage{array,blkarray,multirow,booktabs,longtable}

\usepackage{mymacroschober}

\let\oldtocsubsection=\tocsubsection
\renewcommand{\tocsubsection}[3]{\hspace{0.7em}\oldtocsubsection{#1}{\footnotesize{#2}}{\footnotesize{#3}} }

\usepackage[final]{optional}

\usepackage[colorlinks,hyperfootnotes=false]{hyperref} 
\newcommand{\hrefsf}[2]{\href{#1}{\textsf{#2}}}

\usepackage{versions}
\includeversion{arXiv}

\newcommand{\equalsSep}{0.8pt}

\newcommand{\marginparstretch}{0.6}
\let\oldmarginpar\marginpar
\renewcommand\marginpar[1]{\-\oldmarginpar[\framebox{\setstretch{\marginparstretch}\begin{minipage}{\marginparwidth}{\raggedleft\tiny #1}\end{minipage}}]{\framebox{\setstretch{\marginparstretch}\begin{minipage}{\marginparwidth}{\raggedright\tiny #1}\end{minipage}}}}

\usepackage{tikz,tikz-cd,mathrsfs}
\usetikzlibrary{arrows,positioning,calc}

\def\arrowGap{5pt} 

\usepackage{pgfplots}
\pgfplotsset{compat=1.8}
\usetikzlibrary{decorations.markings}

\newcounter{envcount} \numberwithin{envcount}{section}
\newcounter{keyenvcount}

\newtheorem{theorem}[envcount]{Theorem}
\newtheorem{corollary}[envcount]{Corollary}

\newtheorem{keytheorem}[keyenvcount]{Theorem}
\newtheorem{keycorollary}[keyenvcount]{Corollary}
\newtheorem{keyproposition}[keyenvcount]{Proposition}

\newtheorem*{questions}{Questions}

\theoremstyle{definition} 
\newtheorem{proposition}[envcount]{Proposition}
\newtheorem*{lemma}{Lemma}
\newtheorem*{definition}{Definition}

\newtheorem*{claim}{Claim}

\newtheorem{conjecture}{Conjecture}

\newtheorem{example}[envcount]{Example}
\newtheorem{keyexample}{Example}
\newtheorem*{notation}{Notation}

\theoremstyle{remark}
\newtheorem*{remark}{Remark}
\newtheorem*{acknowledgements}{Acknowledgements}
\newtheorem*{conventions}{Conventions}
\newtheorem*{structure}{Structure of paper}

\def\labelAdjust{0.3}
\def\fillColor{black!20}

\def\flopSide{-}
\def\pr{p}
\def\excpr{q}
\def\mpr{\mmp}
\def\torusProj{\pi}
\def\catquot{\pi}
\def\idem{\mathrm{idem}}
\def\op{\mathrm{op}}

\newcommand\cotangent[1]{\Omega_{#1}}
\newcommand\contactInfinity[1]{\Omega^\infty_{#1}}
\newcommand\cotangentFibre[2]{\Omega_{#1,#2}}

\newcommand\Ladj{\mathsf{LA}}
\newcommand\Radj{\mathsf{RA}}

\def\Hom{\mathop{\rm Hom}\nolimits}

\def\mod{\operatorname{\bf{Mod}}}

\def\Idmatrix{\mathds{1}}

\def\pt{{\rm{pt}}}

\newcommand\compose{\circ}
\newcommand\placeholder{-}

\newcommand\KS{Kapranov--Schechtman}
\newcommand\KandS{Kapranov and Schechtman}
\newcommand\BKS{Bondal--Kapranov--Schechtman}
\newcommand\BKandS{Bondal, Kapranov, and Schechtman}

\newcommand\GPS{Ganatra--Pardon--Shende}
\newcommand\GPandS{Ganatra, Pardon, and Shende}

\newcommand\cO{\mathcal{O}}

\newcommand\cE{\mathcal{E}}
\newcommand\cF{\mathcal{F}}

\newcommand\Ls{{\Lambda_\Sigma}}
\newcommand\Lsinfty{{\Lambda_\Sigma^\infty}}

\newcommand\mmp{{\reflectbox{\it{p}}}}

\renewcommand{\SS}{{\mathop{\mathrm{SS}}}}

\newcommand{\LambdaW}{\bigcup\Lambda_\pm}
\newcommand{\LambdaWsmall}{\bigcup\!\Lambda_\pm}

\newcommand{\window}{\cP_0}
\newcommand{\windowLetter}{\cP}
\newcommand{\microsky}{\cF}

\def\hairLength{0.04}
\def\flagShaded{0}
\def\flagMinus{1}
\def\flagPlus{2}
\def\flagUnion{3}
\def\flagShadedOther{4}

\newcommand{\surfpicture}[1]{
\def\flag{#1}
\begin{tikzpicture}[scale=3.5]
\coordinate (A00) at (0,0);
\coordinate (A0H) at (0,0.5);
\coordinate (A1H) at (1,0.5);
\coordinate (A10) at (1,0);
\coordinate (A01) at (0,1);
\coordinate (A11) at (1,1);

\ifx\flag\flagShaded
\draw[fill,color=\fillColor] (A0H) -- (A10) -- (A01) -- cycle;
\fi

\ifx\flag\flagShadedOther
\draw[fill,color=\fillColor] (A1H) -- (A10) -- (A01) -- cycle;
\fi

\draw[semithick] (A00) -- (A10) -- (A11) -- (A01) -- cycle;

\ifx\flag\flagMinus \else
\draw[blue,semithick] (A10) -- (A01);
\def\steps{20} 
\foreach \x in {3,...,\steps} {
   \coordinate (x) at ($(A10)!1/\steps*(\x-1)!(A01)$); 
   \draw[blue,semithick] (x) -- ($(x)-(45:\hairLength)$); 
   }
\ifx\flag\flagUnion \node at ($(A10)!0.65!(A01)+(-3*\hairLength,-\hairLength)$) {\footnotesize $ \microsky_-$}; \fi
\fi

\draw[semithick] (A10) -- (A0H);
\def\steps{16} 
\foreach \x in {3,...,\steps} {
   \coordinate (x) at ($(A10)!1/\steps*(\x-1)!(A0H)$); 
   \draw[semithick] (x) -- ($(x)-(90-45/2:\hairLength)$); 
   }

\draw[semithick] (A1H) -- (A01);
\def\steps{16} 
\foreach \x in {\ifx\flag\flagPlus 3 \else 4 \fi,...,\steps} {
   \coordinate (x) at ($(A1H)!1/\steps*(\x-2)!(A01)$); 
   \draw[semithick] (x) -- ($(x)-(90-45/2:\hairLength)$); 
   }

\def\steps{14} 
\def\omitA{9}
\def\omitB{8}
\foreach \x in {3,...,\steps} {
   \coordinate (x) at ($(A11)!1/\steps*(\x-1)!(A10)$); 
   \ifx\x\omitA \else \ifx\x\omitB \else
      \draw[semithick] (x) -- ($(x)-(0:\hairLength)$) \fi \fi; 
   \ifx\flag\flagPlus 
      \draw[semithick] (x) -- ($(x)-(0:\hairLength)$) \fi; 
   }

\def\spray{5}
\def\sprayRange{5}
\foreach \x in {1,...,\sprayRange} {
   \coordinate (x) at (A11); 
   \draw[semithick] (x) -- ($(x)-(90/\spray*\x:3*\hairLength)$); 
   }

\ifx\flag\flagPlus \else
\def\sprayRange{4}
\foreach \x in {0,...,\sprayRange} {
   \coordinate (x) at (A1H); 
   \draw[semithick,red] (x) -- ($(x)-(90/\spray*\x:3*\hairLength)$); 
   }
\ifx\flag\flagUnion \node at ($(A1H)+(-4*\hairLength,-3*\hairLength)$) {\footnotesize $\microsky_+$}; \fi
\fi
   
\end{tikzpicture}}

\begin{document}

\title[Mirror symmetry for perverse schobers]{Mirror symmetry for perverse schobers from birational geometry} 
\author{W.\ Donovan}
\address{W.\ Donovan, Yau Mathematical Sciences Center, Tsinghua University, Haidian District, Beijing 100084, China.}
\email{donovan@mail.tsinghua.edu.cn}
\author{T.\ Kuwagaki}
\address{T.\ Kuwagaki, Kavli IPMU (WPI), UTIAS, University of Tokyo, Kashiwa, Chiba 277-8583, Japan.}
\email{tatsuki.kuwagaki@ipmu.jp}

\begin{abstract}
Perverse schobers are categorical analogs of perverse sheaves. Examples arise from varieties admitting flops, determined by diagrams of derived categories of coherent sheaves associated to the flop: in this paper we construct mirror partners to such schobers, determined by diagrams of Fukaya categories with stops, for examples in dimensions~$2$ and $3$. Interpreting these schobers as supported on loci in mirror moduli spaces, we prove homological mirror symmetry equivalences between them. Our construction uses the coherent-constructible correspondence and a recent result of \GPS{}~\cite{GPS} to relate the schobers to certain categories of constructible sheaves. As an application, we obtain new mirror symmetry proofs for singular varieties associated to our examples, by evaluating the categorified cohomology operators of \BKS{}~\cite{BonKapSch} on our mirror schobers.
\end{abstract}

\thanks{}
\maketitle
\thispagestyle{empty}

\section{Introduction}
Mirror symmetry is a collection of mysterious conjectural relationships between complex and symplectic geometry, inspired by the physics of superstring \mbox{theory}. For certain pairs of a complex geometry~$X$ and a symplectic geometry~$Y$, key predictions are:
\begin{enumerate}
\item that a stringy K\"ahler moduli space $\cM_{\text{K\"ah}}$ for~$X$ is isomorphic to a complex structure moduli space $\cM_{\text{CS}}$ for~$Y$, and 
\item an equivalence between a derived category of coherent sheaves on~$X$ and a Fukaya category for~$Y$, called {\em homological mirror symmetry}.
\end{enumerate}
Unifying these, mirror symmetry predicts:
\begin{enumerate}
\setcounter{enumi}{2}
\item an equivalence between a locally constant family of derived categories on~$\cM_{\text{K\"ah}}$, and a corresponding family of Fukaya categories on~$\cM_{\text{CS}}$.
\end{enumerate}

We interpret a `locally constant family of categories' here as an appropriate categorical analog of a locally constant sheaf. Such sheaves arise as solution sheaves for ordinary differential equations: when these equations have singularities, it is natural to study them using certain generalized objects called {\em perverse sheaves}. 

\KandS{} have now suggested categorical analogs of perverse sheaves~\cite{KS2}, named {\em perverse schobers}, or simply {\em schobers}. As perverse sheaves may be thought of as singular versions of locally constant sheaves, schobers give a notion of a locally constant families of categories with singular behaviour.

In the mirror symmetry situation, the locally constant family of categories on $\cM_{\text{K\"ah}}$ is expected to have singular behaviour at certain boundary points. We focus here on boundary points known as {\em conifold points}. Recent research suggests that one can extend the family as a schober over such points by using categories from birational geometry. By mirror symmetry, this schober should then have a mirror partner. Namely, we have: 
\begin{questions} For the locally constant family of Fukaya categories on $\cM_{\text{\em CS}}$,
\renewcommand{\theenumi}{\roman{enumi}}
\begin{enumerate}
\item\label{question a} does this family extend to a perverse schober?

\item\label{question b} does this schober satisfy a mirror symmetry equivalence?
\end{enumerate}
\renewcommand{\theenumi}{\arabic{enumi}}
\end{questions}
In this paper, we give affirmative answers to these questions for some examples, showing that appropriate extensions may be elegantly formulated using Fukaya categories with stops. We apply this to give a new proof of homological mirror symmetry for singularities associated to these examples, by evaluating a categorified cohomology operator on the resulting schober of Fukaya categories.

\subsection{Background}\label{background} We give background on perverse schobers and mirror symmetry, before explaining our results in Section~\ref{section results}.

\subsubsection*{Stringy K\"ahler moduli heuristic} In this paper, we take $X$ to be a resolution of a surface or $3$-fold quadric cone. As a heuristic for these examples, we take $\cM_{\text{K\"ah}}$ to be a punctured disk~$\Delta - p$, and $p$ to be a conifold point. This should be thought of as a local slice of the full stringy K\"ahler moduli: for further physical discussion in the $3$-fold case, see Aspinwall~\cite[Section~4]{Asp}. We therefore construct schobers on $\Delta$, singular at $p$, to answer Question~(\ref{question a}): this will mean constructing a {\em spherical pair}, as follows.

\begin{figure}[h]
\begin{center}
\begin{tikzpicture}
 \node at (0,0) {
\begin{tikzpicture}[node distance=1cm, auto, line width=0.5pt]

\draw[line width=0.5pt] (0,0) circle (1.5);

\draw[fill=black] (0,0) circle (.2ex);
\draw[fill=black] (-1.5,0) circle (.2ex);
\draw[fill=black] (+1.5,0) circle (.2ex);

\draw (0,0) to [bend right=15] (1.5,0);
\draw (0,0) to [bend right=15] (-1.5,0);

 \node (labelSkel) at (0.8,-0.4) {$K$};
 \node (labelSkel) at (1.2,1.5) {$\Delta$};

 \node (pos) at (-1.8,0.2)  {\scriptsize $P_-$};
 \node (pos) at (0.2,0.2)  {\scriptsize $P_0$};
 \node (neg) at (+1.8,0.2) {\scriptsize $P_+$};

\end{tikzpicture}
 };
 \node at (6,0) {
\begin{tikzpicture}[scale=1.5,node distance=1cm, auto, line width=0.5pt]

 \node (mid) at (0,0) {$\cP_0^{\vphantom{+}}$};
 \node (pos) at (-1,0)  {$\cP_-^{\vphantom{+}}$};
 \node (neg) at (+1,0) {$\cP_+^{\vphantom{+}}$};

\draw[right hook->] (pos) to (mid);
\draw[left hook->] (neg) to (mid);
\end{tikzpicture}
 };
 \end{tikzpicture}
\end{center}
\caption{Data determining perverse sheaf, and spherical pair.}\label{figure sph pair}
\end{figure}
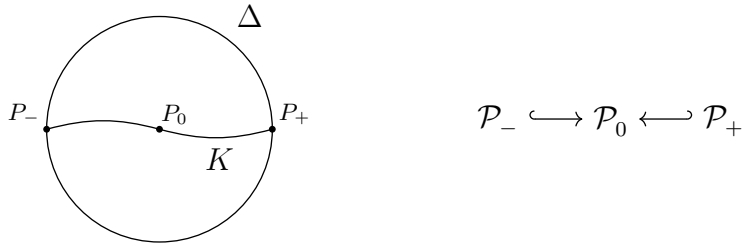

\subsubsection*{Spherical pairs} \KandS{}~give different categorifications of a perverse sheaf~$P$ on $\Delta$, singular at $p$, for different \mbox{{\em skeletons} $K \subset \Delta$~\cite{KS2}}. Take $K$ a path which passes through $0$ as in Figure~\ref{figure sph pair}. Then the sheaf of local cohomology of~$P$ with support in~$K$ is concentrated in some fixed degree (this important property is known as {\em purity}), and is constructible. Furthermore, because~$P$ is singular only at~$p$, this sheaf has only $3$~distinct stalks ($P_\pm$~and~$P_0$). These stalks, along with natural maps between them, turn out to determine~$P$.

This motivates the definition of a spherical pair, illustrated in Figure~\ref{figure sph pair}, as the data of triangulated categories $\cP_\pm$ and $\cP_0$, along with embeddings satisfying certain conditions, and further conditions on orthogonals to these embeddings: for details, see Section~\ref{section schober defs}.

Examples of spherical pairs arise from birational geometry. Letting $X_\pm$ be the two sides of a $3$-fold flop, or more generally certain flops of families of curves, Bodzenta and Bondal~\cite{BB} construct a spherical pair given by the following data. 
\begin{equation*}
 \begin{tikzpicture}[scale=1.7]
\node (Bminus) at (-1.3,1) {$ D(X_-) $};
\node (Bzero) at (0,1) {$ \cQ_0 $};
\node (Bplus) at (+1.3,1) {$ D(X_+) $};
\draw[right hook->] (Bminus) to node[above] {} (Bzero); 
\draw[left hook->] (Bplus) to node[above] {} (Bzero); 
\end{tikzpicture}
\end{equation*}
Here $\cQ_0$ is an appropriate quotient of the derived category~$D(X_B)$ of the fibre product of $X_\pm$ over their common contraction. In this paper we give an alternative construction using a subcategory of $D(X_B)$, in Theorem~\ref{theorem schober equivalence} below.

\subsubsection*{Flobers} Bondal, Kapranov, and Schechtman give related constructions for certain webs of flops, under the name of {\em flobers}, which may be viewed as categorified perverse sheaves on $\mathbb{C}^n$, singular along a real hyperplane arrangement~\cite{BonKapSch}. The prototype of a flober (on $\mathbb{C}$, singular at $0$) is determined by data as follows.
\begin{equation}\tag{$\ast$}\label{equation flober B}
 \begin{tikzpicture}[scale=1.8,baseline={([yshift=-.5ex]current bounding box.center)}]
\node (Bminus) at (-1.3,1) {$ D(X_-) $};
\node (Bzero) at (0,1) {$ D(X_B) $};
\node (Bplus) at (+1.3,1) {$ D(X_+) $};
\draw[right hook->] (Bminus) to node[above] {} (Bzero); 
\draw[left hook->] (Bplus) to node[above] {} (Bzero); 
\end{tikzpicture}
\end{equation}
A flober is defined by strictly weaker conditions than a spherical pair: in this case, these conditions amount to the usual Fourier--Mukai functors between $ D(X_-) $ and $ D(X_+) $ being equivalences.

\subsubsection*{Mirror symmetry equivalences} Our version of homological mirror symmetry combines work of the second author~\cite{K2} and \GPandS~\cite{GPS}. For a large class of toric stacks~$X$, this gives an equivalence between the bounded derived category of coherent sheaves~$D(X)$ and a certain Fukaya category. Namely, we take a torus $T$ obtained from the (dual of the) toric data for $X$, and consider the {\em wrapped Fukaya category} $\cW_{\Lambda^\infty}(\cotangent{T})$ with {\em stop} $\Lambda^\infty$, a locus in contact infinity of the cotangent bundle~$\cotangent{T}$ which Lagrangians in the category must avoid, also determined by the toric data of~$X$. Further details are given in Section~\ref{section method}.

\subsection{Results}\label{section results} Take $X_0$ to be one of the following singularities, along with two (stacky) resolutions $X_\pm$ as described. 

\begin{keyexample}\label{keyexample surface} Take $X_0$ the quotient of $\mathbb{C}^2$ by $\mathbb{Z}/2\mathbb{Z}$ acting by $\pm 1$. Then let $X_-$~be the associated Deligne--Mumford quotient stack, and $X_+$~be the minimal resolution of the singularity~$X_0$.
\end{keyexample}

\begin{keyexample}\label{keyexample threefold} Take $X_0$ the conifold singularity $\{ xy-zw = 0 \},$ and $X_\pm$ two small resolutions related by an Atiyah flop.
\end{keyexample}

\noindent We denote the fibre product $X_B$ of $X_\pm$ over the singularity $X_0$, and associated morphisms, as follows: see Section~\ref{section schober constructions} for an explicit construction of~$X_B$.
\[
\begin{tikzpicture}[xscale=1,scale=1.25]

\node (Bminus) at (0,1) {$ X_B $};
\node (Bplus) at (1,0) {$ X_+ $};

\node (Cminus) at (-1,0) {$ X_- $};
\node (Cplus) at (0,-1) {$ X_0 $};

\draw[->] (Bminus) to node[above right] {\scriptsize $\pr_+$}  (Bplus);
\draw[->] (Cminus) to (Cplus);

\draw[<-] (Cminus) to node[above left] {\scriptsize $\pr_-$}  (Bminus);
\draw[<-] (Cplus) to (Bplus);

\end{tikzpicture}
\]

\subsubsection*{Spherical pairs} Following our heuristic above, we take $\cM_{\text{CS}}$ and $\cM_{\text{K\"ah}}$ to be punctured disks~$\Delta - p$. The following is our main theorem. It constructs perverse schobers on partial compactifications $\overline{\cM}_{\text{CS}}$ and $\overline{\cM}_{\text{K\"ah}}$, both identified with the disk~$\Delta$. We give these schobers in the form of spherical pairs, and show they are equivalent, answering Questions~(\ref{question a}) and~(\ref{question b}) for our examples.

\begin{keytheorem}[Theorems~\ref{theorem mirror Fukaya}, \ref{theorem mirror Fukaya surf}]\label{theorem schober equivalence}
For Examples~\ref{keyexample surface} and~\ref{keyexample threefold}, we have
\renewcommand{\theenumi}{\roman{enumi}}
\begin{enumerate}
\item schobers given by data below, and

\item an equivalence of these schobers via homological mirror symmetry, explained in Section~\ref{section method}.
\end{enumerate}
\renewcommand{\theenumi}{\arabic{enumi}}

\smallskip
\noindent{\em\bf Symplectic side:}
Let $T$ be a torus and $\Lambda_{\pm}^\infty$ be loci in contact infinity of~$\cotangent{T}$, all determined by the toric data of $X_\pm$ in Section~\ref{section schober constructions}. Then take data as follows, with embeddings given in the course of the proof.
\begin{equation}\tag{A}\label{equation sph pair A}
 \begin{tikzpicture}[scale=2.3,baseline={([yshift=-.5ex]current bounding box.center)}]
\node (Bminus) at (-1.5,1) {$ \cW_{\Lambda_-^\infty}(\cotangent{T}) $};
\node (Bzero) at (0,1) {$ \cW_{\Lambda_-^\infty \cup \Lambda_+^\infty}(\cotangent{T}) $};
\node (Bplus) at (+1.5,1) {$ \cW_{\Lambda_+^\infty}(\cotangent{T}) $};
\draw[right hook->] (Bminus) to node[above] {} (Bzero);
\draw[left hook->] (Bplus) to node[above] {} (Bzero); 
\end{tikzpicture}
\end{equation}
\noindent{\em\bf Complex side:}
Let  $\window$ be the subcategory of $ D(X_B) $ generated by the images of the embeddings of $ D(X_\pm) $ in~\eqref{equation flober B}. Then take data as follows.
\begin{equation}\tag{B}\label{equation sph pair B}
 \begin{tikzpicture}[scale=1.7,baseline={([yshift=-.5ex]current bounding box.center)}]
\node (Bminus) at (-1.3,1) {$ D(X_-) $};
\node (Bzero) at (0,1) {$ \window $};
\node (Bplus) at (+1.3,1) {$ D(X_+) $};
\draw[right hook->] (Bminus) to node[above] {} (Bzero); 
\draw[left hook->] (Bplus) to node[above] {} (Bzero); 
\end{tikzpicture}
\end{equation}
\end{keytheorem}

\begin{remark}The mirror operation to the $3$-fold flop in Example~\ref{keyexample threefold} is given by Fan, Hong, Lau, and~Yau~\cite{FHLY}. It would be interesting to relate this to the locally constant family of Fukaya categories on the punctured disk~$\Delta - p$ which arises from Theorem~\ref{theorem schober equivalence}.
\end{remark}

\begin{remark}It would be interesting to try to use the schobers appearing in Theorem~\ref{theorem schober equivalence} to categorify perverse sheaves extending the A-model and B-model local systems of cohomology on $\cM_{\text{CS}}$ and $\cM_{\text{K\"ah}}$, in our examples and more generally. In particular, this could lead to an {\em a priori} reason why the derived equivalences associated to a flop may be organized into a categorified perverse sheaf, as asked by \BKS{}~\cite[Section~0A]{BonKapSch}.
\end{remark}

\subsubsection*{Symplectic flobers} Some of the techniques used to prove Theorem~\ref{theorem schober equivalence} then yield the following, giving a counterpart to the flober~\eqref{equation flober B} on the symplectic side.

\begin{keytheorem}[Theorem~\ref{theorem flober equiv}]\label{keytheorem flober equiv} For Examples~\ref{keyexample surface} and~\ref{keyexample threefold}, the flober~\eqref{equation flober B} is equivalent to a flober as follows, where $\Lambda_B^\infty$ is a locus in contact infinity of $\cotangent{T}$, determined by the toric data of $X_B$.
\begin{equation*}\label{equation flober A}
 \begin{tikzpicture}[scale=2.2]
\node (Bminus) at (-1.5,1) {$ \cW_{\Lambda_-^\infty}(\cotangent{T}) $};
\node (Bzero) at (0,1) {$ \cW_{\Lambda_B^\infty}(\cotangent{T}) $};
\node (Bplus) at (+1.5,1) {$ \cW_{\Lambda_+^\infty}(\cotangent{T}) $};
\draw[right hook->] (Bminus) to node[above] {} (Bzero); 
\draw[left hook->] (Bplus) to node[above] {} (Bzero); 
\end{tikzpicture}
\end{equation*}
\end{keytheorem}

\subsubsection*{Application} Perverse sheaves admit certain cohomology operators which yield vectors spaces: these are used, for instance, to define intersection cohomology. By analogy, \BKandS{}~\cite{BonKapSch} define certain categorified cohomology operators on perverse schobers which yield categories. In particular, they define the 2nd cohomology with compact support $\mathbb{H}^2_c$ to be a homotopy push-out of the diagram defining the flober.  They prove that~$\mathbb{H}^2_c$ for~\eqref{equation flober B} is~$D(X_0)$, with~$X_0$ the singular base given above. We~prove a symplectic analog of this result, as follow.

\begin{keyproposition}[Proposition~\ref{proposition a-side cohom}]\label{keyproposition a-side cohom}
The 2nd cohomology with compact support $\mathbb{H}^2_c$ for the symplectic flober in Theorem~\ref{keytheorem flober equiv} is given by the category
\[
\cW_{\Lambda_-^\infty \cap \Lambda_+^\infty}(\Omega_T).
\]
\end{keyproposition}

\noindent We immediately obtain the following statement of homological mirror symmetry for the singular space~$X_0$.

\begin{keycorollary}[Corollary~\ref{cor sing equiv}]\label{keycorollary sing equiv} The equivalence of Theorem~\ref{keytheorem flober equiv} induces an equivalence
\begin{equation*}\label{equation sing equiv}
D(X_0) \cong \cW_{\Lambda_-^\infty \cap \Lambda_+^\infty}(\Omega_T).
\end{equation*}
\end{keycorollary}

An equivalence of the categories in Corollary~\ref{keycorollary sing equiv} follows from work of the second author in~\cite{K2}, as the methods there apply in singular cases. The corollary provides a new alternatively proof of this, which uses the results of~\cite{K2} only for smooth cases, along with the categorified cohomology operator~$\mathbb{H}^2_c$.

\begin{remark}\BKS{} also define 1st cohomology~$\mathbb{H}^1$ of a flober~\cite[Section~2D]{BonKapSch}: for the flober~\eqref{equation flober B} this is by definition a quotient of $D(X_B)$ by the triangulated subcategory generated by the embeddings of~$D(X_\pm)$. In this language, the spherical pairs of Theorem~\ref{theorem schober equivalence} are special in the sense that their~$\mathbb{H}^1$ is zero, whereas the flober~\eqref{equation flober B} and the flober in Theorem~\ref{keytheorem flober equiv} have non-zero~$\mathbb{H}^1$.
\end{remark}

\subsection{Method of proof}\label{section method} Let $X$ be one of the toric varieties or stacks appearing above, and let $T$ be the corresponding dual torus: namely, if $X$ is described by a toric fan $\Sigma$ in $N_\bR$ where $N=\Hom(M,\bZ)$ for a lattice $M$, then we take $T=M_\bR/M$. Write $\wSh_{\Lambda}(T)$ for a category of {\em wrapped constructible sheaves} on~$T$, with microsupport~$\Lambda$ in~$\cotangent{T}$, where $\Lambda$ is the {\em FLTZ skeleton} for the toric data. Furthermore let $\Lambda^\infty$ be the associated locus at infinity in~$\cotangent{T}$, where full details are given in Section~\ref{section microlocal fukaya}. Then the homological mirror symmetry equivalences used above are obtained by composition as follows.
\[\cW_{\Lambda^\infty}(\cotangent{T}) \underset{\scriptsize (1)}{\cong} \wSh_{\Lambda}(T)^{\op} \underset{\scriptsize (2)}{\cong} D(X)^{\op} \underset{\scriptsize (3)}{\cong} D(X)\]
\begin{enumerate}
\item\label{enum hms step 1} This is proved in recent work of \GPandS~\cite{GPS}, refining a result of Nadler--Zaslow~\cite{NZ}.
\item\label{enum hms step 2} This is a version of the coherent-constructible correspondence, as proved by the second author in~\cite{K2}.
\item\label{enum hms step 3}  Here we take the derived dual $\RHom(\placeholder,\cO_X)$ to give a covariant composition.
\end{enumerate}

We prove Theorem~\ref{theorem schober equivalence} by showing that the diagrams~(\ref{equation sph pair A}) and~(\ref{equation sph pair B}) appearing there are equivalent to the following diagram (where, for convenience, we drop `$\op$' from the notation).
\begin{equation}\tag{C}\label{equation opp sph pair C}
 \begin{tikzpicture}[scale=2.2,baseline={([yshift=-.5ex]current bounding box.center)}]
\node (Bminus) at (-1.5,1) {$ \wSh_{\Lambda_-}(T) $};
\node (Bzero) at (0,1) {$ \wSh_{\Lambda_- \cup \Lambda_+}(T) $};
\node (Bplus) at (+1.5,1) {$ \wSh_{\Lambda_+}(T) $};
\draw[right hook->] (Bminus) to node[above] {} (Bzero);
\draw[left hook->] (Bplus) to node[above] {} (Bzero);
\end{tikzpicture}
\end{equation}
For the symplectic side diagram~(\ref{equation sph pair A}), this equivalence is formal, following from~(\ref{enum hms step 1}) and by construction. For the complex side~(\ref{equation sph pair B}), we proceed as follows. We first apply the equivalence~(\ref{enum hms step 2}) to the spaces which appear in the flober~\eqref{equation flober B}. From these we deduce that there is an embedding \[\kappa \colon \window \hookrightarrow \wSh_{\Lambda_- \cup \Lambda_+}(T).\]
We prove that $\kappa$ is an equivalence by constructing and comparing certain {\em semiorthogonal decompositions} of these categories, as follows. To illustrate our proof, and for interest, these decompositions are presented in Section~\ref{section outline} for the surface case.
\begin{itemize}
\item A decomposition of $\wSh_{\Lambda_- \cup \Lambda_+}(T)$ with a component~$\wSh_{\Lambda_\pm}(T)$ arises from an explicit presentation of the skeleton~$\Lambda_- \cup \Lambda_+$.
\item A decomposition of~$\window$ with a component~$D(X_\pm)$ is standard in the \mbox{$3$-fold}~case. For the surface case, we construct one by studying derived categories of appropriate GIT quotients.
\end{itemize}
We thus establish that $\kappa$ is an equivalence. It follows from the decompositions of~$\window$ that the diagram~(\ref{equation sph pair B}) gives a spherical pair, and we thence deduce that~(\ref{equation opp sph pair C}) gives a spherical pair. That the diagram~(\ref{equation sph pair A}) gives a spherical pair then follows by composing with the equivalence~(\ref{enum hms step 1}).

Theorem~\ref{keytheorem flober equiv} is proved by a similar, but simpler, argument using a diagram of wrapped constructible categories as follows.
\begin{equation*}\label{equation flober C}
 \begin{tikzpicture}[scale=2.2,baseline={([yshift=-.5ex]current bounding box.center)}]
\node (Bminus) at (-1.5,1) {$ \wSh_{\Lambda_-}(T) $};
\node (Bzero) at (0,1) {$ \wSh_{\Lambda_B}(T) $};
\node (Bplus) at (+1.5,1) {$ \wSh_{\Lambda_+}(T) $};
\draw[right hook->] (Bminus) to node[above] {} (Bzero); 
\draw[left hook->] (Bplus) to node[above] {} (Bzero); 
\end{tikzpicture}
\end{equation*}

To prove Proposition~\ref{keyproposition a-side cohom} we adapt a method of \BKandS{} to show that $\mathbb{H}^2_c$ applied to the the diagram above gives $ \wSh_{\Lambda_0}(T) $, using explicit presentation of skeleta. The result follows again by composing with the equivalence~(\ref{enum hms step 1}).

\begin{remark} In the course of the proof of Proposition~\ref{keyproposition a-side cohom} we obtain a new proof of an instance of the coherent constructible correspondence, namely
\[
D(X_0) \cong \wSh_{\Lambda_0}(T).
\]
At the end of paper, we present a conjectural picture about how this method could be extended to prove more general such results.
\end{remark}

\subsection{Related work} Nadler has also discussed mirror equivalences of schobers on the disk~\cite{N2}. In this case a different skeleton~$K$ is used so that the schober takes the form of a {\em spherical functor}. He proves a homological mirror symmetry statement relating a certain Landau--Ginzburg A-model to the B-model for the higher-dimensional pair of pants~\cite[Corollary~1.5]{N2} by deducing it from a mirror equivalence of such schobers~\cite[Theorem~1.4]{N2}. The A-model schober in this case is over a disk in the space of values of the (complex) superpotential.

Harder and Katzarkov have used schobers to given a new proof of a mirror equivalence for the projective space~$\bP^3$~\cite{HarKat}.

The first author previously constructed schobers on the complex side of mirror symmetry, associated to wall crossings in GIT~\cite{Don}. In subsequent work he applied these to study schobers associated to standard flops, which led to a discussion of the Questions~(\ref{question a}) and~(\ref{question b})~\cite[Section~1.2]{Don2}. In forthcoming work~\cite{DW5}, the first author
 and M.~Wemyss give a mathematical treatment of the stringy K\"ahler moduli for general $3$-fold flops of irreducible curves.

\begin{structure}
In \S\ref{section outline} we explain some details of our surface example informally. In \S\ref{section preliminaries} we give preliminaries. In \S\ref{section schober constructions} we construct the data of spherical pairs and flobers, and in \S\ref{section schober equiv} we prove that they satisfy the required properties, and give homological mirror symmetry equivalences between them, proving Theorems~\ref{theorem schober equivalence} and~\ref{keytheorem flober equiv}. Finally, in \S\ref{section sing HMS} we prove Proposition~\ref{keyproposition a-side cohom}, deduce Corollary~\ref{keycorollary sing equiv}, and finish with some remarks and conjectures on further applications of flobers to the coherent-constructible correspondence.
\end{structure}

\begin{acknowledgements} Both authors are grateful for conversations with A.\ Bondal and M.\ Kapranov at Kavli IPMU. The first author is grateful for discussions with T.\ Logvinenko, and thanks D.\ Auroux, T.\ Coates, D.\ Nadler for helpful conversations at an early stage of the project. The authors also thank the organizers of the conference `Categorical and Analytic Invariants in Algebraic Geometry V' at Osaka in 2018, where this collaboration started.

The two authors acknowledge the support of WPI Initiative, MEXT, Japan, and of JSPS KAKENHI Grants~JP16K17561 and~JP18K13405 respectively. The first author is supported by the Yau MSC, Tsinghua University, and the Thousand Talents Plan.
\end{acknowledgements}

\begin{conventions} We denote categories as follows: for details, see Section~\ref{section microlocal fukaya}.

\smallskip
\begin{tabular}{p{1cm} l}
$\cSh$ & constructible sheaves \\
$\lSh$ & quasi-constructible sheaves \\
$\wSh$ & wrapped constructible sheaves~\cite{N} \\
$D$ & bounded derived category of coherent sheaves
\end{tabular}
\end{conventions}

\section{Surface example}
\label{section outline}

In this section we present semiorthogonal decompositions for the surface case, to illustrate our proof, and for interest. 

\subsection{Setting} We take the singularity $ X_0 = \bC^2/\bZ_2 $, and take (stacky) resolutions $ X_\pm$ given by the corresponding Deligne--Mumford stack~$X_-$ and the minimal resolution~$X_+$, denoted as below. The fibre product~$ X_B $ is a further  Deligne--Mumford stack, given explicitly in Section~\ref{section surface geometry}.
\[
\begin{tikzpicture}[xscale=1,scale=1.25]

\node (Bminus) at (0,1) {$ X_B $};
\node (Bplus) at (1,0) {$ X_+ $};
\node (Dplus) at (2.5,0) {\phantom{$ \widetilde{\bC^2/\bZ_2} $}};
\node (DplusLabel) at (2.5,0.05) {$ \widetilde{\bC^2/\bZ_2} $};

\node (Cminus) at (-1,0) {$ X_- $};
\node (Dminus) at (-2.5,0) {$ [\bC^2/\bZ_2] $};

\node (Cplus) at (0,-1) {$ X_0 $};

\draw[->] (Bminus) to  (Bplus);  
\draw[->] (Cminus) to node[below left] {\scriptsize $f_-$} (Cplus);

\draw[<-] (Cminus) to (Bminus); 
\draw[<-] (Cplus) to node[below right] {\scriptsize $f_+$} (Bplus);

\draw[-,transform canvas={yshift=+\equalsSep}] (Dminus) to (Cminus);
\draw[-,transform canvas={yshift=-\equalsSep}] (Dminus) to (Cminus);

\draw[-,transform canvas={yshift=+\equalsSep}] (Dplus) to (Bplus);
\draw[-,transform canvas={yshift=-\equalsSep}] (Dplus) to (Bplus);

\end{tikzpicture}
\]

The fibres of $f_-$ and $f_+$ over the singularity~$0$ are the stacky point~$[0/\bZ_2]$, and a projective line~$\bP^1$. Let $\cE_-$ and $\cE_+$ be the structure sheaves of these fibres, tensored by the non-trivial irreducible representation of~$\bZ_2$, and the twisting sheaf $\cO(-1)$, respectively.

\def\labelShift{-0.48}
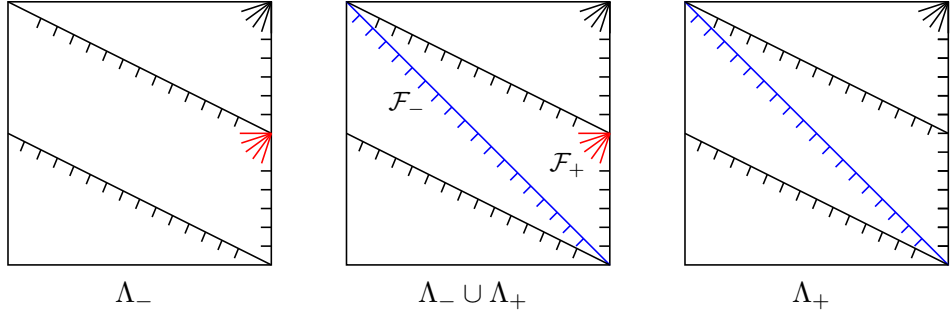
\begin{figure}[h]
\[
\begin{tikzpicture}[scale=4.5]
\node at (-1,0) {\surfpicture{1}};
\node at (0,0) {\surfpicture{3}};
\node at (1,0) {\surfpicture{2}};
\node at (-1,\labelShift) {\small $\Lambda_-$};
\node at (0,\labelShift) {\small $\Lambda_- \cup \Lambda_+$};
\node at (1,\labelShift) {\small $\Lambda_+$};
\end{tikzpicture}
\]
\caption{Skeleta for symplectic side schober}
\label{figure surface skeleta}
\end{figure}

Now take skeleta $\Lambda_\pm$ determined by the toric data of $X_\pm$, as illustrated in Figure~\ref{figure surface skeleta}, which we explain now. The construction of these skeleta is discussed in Section~\ref{section surface schober equiv}. The squares in the figure show a fundamental domain for the quotient $T=\bR^2/\bZ^2$. The `hairs' denote cotangent directions in~$\Omega_T$, and thereby indicate conical loci in~$\Omega_T$ corresponding to the skeleton~$\Lambda$. We~let $\microsky_-$ and $\microsky_+$ be microlocal skyscraper sheaves~\cite{N} in~$\wSh_{\Lambda_- \cup \Lambda_+}(T)$ corresponding to the cotangent directions marked.

\begin{keyproposition}[Propositions~\ref{prop schober surface}, \ref{sod}]
The category $\cP_0$ appearing in Theorem~\ref{theorem schober equivalence} has semiorthogonal decompositions as follows.
\begin{equation*}
 \la D(X_-), \cE_- \ra = \cP_0 = \la D(X_+), \cE_+ \ra
\end{equation*}
\noindent\smallskip The category $\wSh_{\Lambda_- \cup \Lambda_+}(T)$ has decompositions as follows.
\begin{equation*}
 \la \wSh_{\Lambda_-}(T), \microsky_- \ra = \wSh_{\Lambda_- \cup \Lambda_+}(T) = \la \wSh_{\Lambda_+}(T), \microsky_+ \ra
\end{equation*}
\end{keyproposition}

In the course of the proof we see that these categories, and the components in the above decompositions, all correspond under the coherent-constructible correspondence.

The embeddings of the $D(X_\pm)$ are the pullbacks $p_{\pm *}$, and it follows that the equivalences between $ D(X_-)$ and $D(X_+) $ required in the definition of a spherical pair follow from simple cases of results of Bridgeland--King--Reid~\cite{BKR}.

\begin{remark} We get decompositions of Fukaya categories
\begin{equation*}
 \big\langle \cG_-, \cW_{\Lambda_-^\infty}(\cotangent{T}) \big\rangle = \cW_{\Lambda_-^\infty \cup \Lambda_+^\infty}(\cotangent{T}) = \big\langle \cG_+, \cW_{\Lambda_+^\infty}(\cotangent{T}) \big\rangle
\end{equation*}
for certain objects $\cG_\pm$ of $\cW_{\Lambda_-^\infty \cup \Lambda_+^\infty}(\cotangent{T})$, by applying the equivalence of~\GPandS{}. (That the categories $\cW$ appear in the decompositions here on the right, whereas the categories $\wSh$ appear on the left in the proposition above, is due to the fact that this equivalence is contravariant.)
\end{remark}

\section{Preliminaries}
\label{section preliminaries}

\subsection{Perverse schobers}
\label{section schober defs}

We work with \emph{spherical pairs}, as follows.
Recall that a semi-orthogonal decomposition
$
\cT = \big\langle \cC, \cD \big\rangle,
$
is determined by embeddings
 \[\gamma \colon \cC \longrightarrow \cT \qquad\text{and}\qquad \delta \colon \cD \longrightarrow \cT,\] 
and induces projection functors given by adjoints as follows.
 \[\gamma^{\Ladj}\colon \cT \longrightarrow \cC \qquad\text{and}\qquad \delta^{\Radj}\colon \cT \longrightarrow \cD. \]
We may then make the following definition.

\begin{definition}\label{definition sph pair}\cite{KS2}
A \emph{spherical pair} $\cP$ is  a triangulated category $\cP_0$ with admissible subcategories $\cP_\pm$ and semi-orthogonal decompositions
\begin{equation*}
\big\langle \cQ_-, \cP_- \big\rangle = \cP_0 = \big\langle \cQ_+, \cP_+ \big\rangle,
\end{equation*}
such that compositions of the embeddings and projections above
 \[ \cQ_- \longleftrightarrow \cQ_+ \qquad\text{and}\qquad \cP_- \longleftrightarrow \cP_+, \]
are equivalences.
\end{definition}

\begin{remark}As indicated in Section~\ref{background}, this data should be thought of as a categorification of vector space data determining a perverse sheaf. Further discussion is given by \KS{} in~\cite[Section~9A]{KS1} and~\cite{KS2}.
\end{remark}

\begin{definition} An isomorphism between spherical pairs~$\cP$ and $\cP'$ consists of equivalences $\cP_\bullet\simeq\cP'_\bullet$ and $\cQ_\bullet\simeq\cQ'_\bullet$ intertwining the embeddings.
\end{definition}

A weaker notion is discussed by \BKandS~\cite{BonKapSch} omitting the condition on orthogonals $\cQ_\bullet$.

\begin{definition}\cite[Section~1B]{BonKapSch}
A \emph{flober} or \emph{weak spherical pair} $\cP$ is a triangulated category  $\cP_0$ with admissible subcategories $\cP_\pm$ with embeddings $\delta_\pm$ such that the 
compositions of $\delta_\pm$ and $\delta_\pm^{\Radj}$ as follows 
 \[ \cP_- \longleftrightarrow \cP_+ \] are equivalences.
\end{definition}

A spherical pair yields a flober in the obvious way.

\begin{example} Flobers arise from~$3$-fold flops and more general flops of curves, see~\cite[Section~1B,~1C]{BonKapSch} for discussion.\end{example}

\BKS{} make following definition.

\begin{definition}\cite[Section~2E]{BonKapSch}
For a weak spherical pair $\cP$, the homology with compact support $\mathbb{H}^2_c(\Delta,\cP)$ is defined as the homotopy push-out in the Morita model of the following.
\begin{equation*}
\begin{tikzpicture}[xscale=1,scale=1.25]

\node (Bminus) at (0,1) {$ \cP_0 $};
\node (Bplus) at (1,0) {$ \cP_+ $};

\node (Cminus) at (-1,0) {$ \cP_- $};

\draw[->] (Bminus) to node[above right] {\scriptsize $\delta_+^{\Radj}$}  (Bplus);
\draw[<-] (Cminus) to node[above left] {\scriptsize $\delta_-^{\Radj}$}  (Bminus);

\end{tikzpicture}
\end{equation*}
\end{definition}

We will often say, for brevity, that a spherical pair $\cP$ is determined by data written as follows, sometimes omitting the adjoint functors.
\begin{equation*}
 \begin{tikzpicture}[scale=2]
\node (Bminus) at (-1,1) {$ \cP_- $};
\node (Bzero) at (0,1) {$ \cP_0 $};
\node (Bplus) at (+1,1) {$ \cP_+ $};
\draw[right hook->,transform canvas={yshift=+\arrowGap}] (Bminus) to node[above] {\scriptsize $\delta_-$} (Bzero);
\draw[left hook->,transform canvas={yshift=+\arrowGap}] (Bplus) to node[above] {\scriptsize $\delta_+$} (Bzero);
\draw[<-,transform canvas={yshift=-\arrowGap}] (Bminus) to node[below] {\scriptsize $\delta_-^{\Radj}$} (Bzero);
\draw[<-,transform canvas={yshift=-\arrowGap}] (Bplus) to node[below] {\scriptsize $\delta_+^{\Radj}$} (Bzero);
\end{tikzpicture}
\end{equation*}
As we will study anti-equivalences, it will also be convenient to say that a spherical pair is determined by the \emph{opposite} of categorical data written as follows, where we reuse the notation $\delta_\pm$ for the induced functors below.
\begin{equation*}
 \begin{tikzpicture}[scale=2]
\node (Bminus) at (-1,1) {$ \cP^{\op}_- $};
\node (Bzero) at (0,1) {$ \cP^{\op}_0 $};
\node (Bplus) at (+1,1) {$ \cP^{\op}_+ $};
\draw[right hook->,transform canvas={yshift=+\arrowGap}] (Bminus) to node[above] {\scriptsize $\delta_-$} (Bzero);
\draw[left hook->,transform canvas={yshift=+\arrowGap}] (Bplus) to node[above] {\scriptsize $\delta_+$} (Bzero);
\draw[<-,transform canvas={yshift=-\arrowGap}] (Bminus) to node[below] {\scriptsize $\delta_-^{\Ladj}$} (Bzero);
\draw[<-,transform canvas={yshift=-\arrowGap}] (Bplus) to node[below] {\scriptsize $\delta_+^{\Ladj}$} (Bzero);
\end{tikzpicture}
\end{equation*}
In particular, given such data we have semi-orthogonal decompositions
\begin{equation*}
\big\langle \cP^{\op}_-, \cQ^{\op}_- \big\rangle = \cP^{\op}_0 = \big\langle \cP^{\op}_+, \cQ^{\op}_+ \big\rangle,
\end{equation*}
and $\mathbb{H}^2_c(\Delta,\cP)$ is given by a push-out of the following.
\begin{equation*}
\begin{tikzpicture}[xscale=1,scale=1.25]

\node (Bminus) at (0,1) {$ \cP^{\op}_0 $};
\node (Bplus) at (1,0) {$ \cP^{\op}_+ $};

\node (Cminus) at (-1,0) {$ \cP^{\op}_- $};

\draw[->] (Bminus) to node[above right] {\scriptsize $\delta_+^{\Ladj}$}  (Bplus);
\draw[<-] (Cminus) to node[above left] {\scriptsize $\delta_-^{\Ladj}$}  (Bminus);

\end{tikzpicture}
\end{equation*}

\subsection{Fukaya category and microlocal sheaf theory}
\label{section microlocal fukaya}
In this subsection, we give a brief introduction to a relationship between the Fukaya category and microlocal sheaf theory.
Let $Z$ be a real analytic manifold and $k_Z$ be the constant sheaf valued in a field $k$. We denote the bounded derived category of $k_Z$-modules by $\mod(k_Z)$. We define microsupport (also known as `singular support'), one of the most important notion in microlocal sheaf theory. 

\begin{definition}
The microsupport $\SS(\cE)$ of $\cE\in \mod(k_Z)$ is a subset of the cotangent bundle $\cotangent{Z}$ defined by its complement as follows: $(x, \xi)\in \cotangent{Z}$ (where $x\in Z$ and $\xi\in \cotangentFibre{Z}{x}$) is \underline{not} contained in $\SS(\cE)$ if there exists an open neighbourhood~$U$ of $(x, \xi)$ such that
\begin{equation*}
\big(\RGamma_{\lc z\relmid \psi(z)\geq \psi(y)\rc}\cE\,\big)_y \simeq 0
\end{equation*}
for any point $y\in Z$ and any smooth function $\psi$ with $\Graph(d\psi)\subset U$.
\end{definition}

We view $\cotangent{Z}$ as a symplectic manifold with its standard exact symplectic structure. The microsupport $\SS(\cE)$ is a Lagrangian subset of $\cotangent{Z}$ if $\cE$ is constructible:
\begin{definition}
A sheaf $\cE$ is constructible (respectively quasi-constructible) if there exists a real analytic Whitney stratification $\cS$ of $Z$ such that the restriction of $\cE$ to each stratum $S\in \cS$ is a locally constant sheaf of finite rank (respectively locally constant sheaf possibly of infinite rank). For the definition of such a stratification, we refer to \cite{sheaveson}.
\end{definition}
\begin{theorem}[Involutivity theorem \cite{sheaveson}]
An object $\cE\in \mod(k_Z)$ is quasi-constructible if and only if $\SS(\cE)$ is Lagrangian.
\end{theorem}
We write $\cSh(Z)$ (respectively~$\lSh(Z)$) for the (dg-)category of bounded complexes of constructible sheaves (respectively  unbounded complexes of quasi-constructible sheaves) over $Z$. For a subset $\Lambda\subset \cotangent{Z}$, the full subcategory of $\cSh(Z)$ (respectively~$\lSh(Z))$ spanned by objects with microsupport living inside $\Lambda$ is denoted by $\cSh_\Lambda(Z)$ (respectively~$\lSh_\Lambda(Z))$.

Before stating a relationship with Fukaya category, we introduce one more category.
\begin{definition}
Let $\Lambda\subset \cotangent{Z}$ be a subset. The full subcategory $\wSh_\Lambda(Z)$ of $\lSh_{\Lambda}(Z)$ is defined by the following: $\cE\in \wSh_\Lambda(Z)$ if and only if 
\[\Hom_{\lSh_{\Lambda}(Z)}(\cE, -)\colon \lSh_{\Lambda}(Z)\rightarrow \mathrm{\bf{Vect}}\] commutes with any direct sums. We call an object of $\wSh_\Lambda(Z)$ a wrapped constructible sheaf.
\end{definition}

A relationship between microlocal sheaf theory and the Fukaya category was first clearly stated by Nadler--Zaslow by using an `infinitesimally wrapped Fukaya category'. Our setting is a more recent variant of Nadler--Zaslow established by \GPS~\cite{GPS}: Let $Z$ be a real analytic manifold and $\Omega_Z$ be the cotangent bundle of $Z$. Let $\{x_i\}$ be local coordinates of $Z$ and $\{\xi_i\}$ be the corresponding cotangent coordinates. Then $\Omega_Z$ has an exact symplectic structure locally written as $\sum_id\xi_i\wedge dx_i$. Let us fix $g$ a Riemannian metric over $Z$. Then we define the cosphere bundle by 
\[
S^*Z_a:=\lc (x, \xi)\in \Omega_Z\relmid g_x(\xi, \xi)=a\rc.
\]
The symplectic structure induces a contact structure on $S^*Z_a$. Since the $S^*Z_a$ for various $a$ are contactoisomorphic to each other, we consider $a\rightarrow \infty$ virtually and call this abstract contact manifold {\em contact infinity} and denote it by $\contactInfinity{Z}$. Let $\Lambda^\infty$ be a subanalytic Legendrian of $\contactInfinity{Z}$. Then one can define an $A_\infty$-category $\cW_{\Lambda^\infty}(\cotangent{Z})$, the wrapped Fukaya category of $\cotangent{Z}$ with the stop $\Lambda^\infty$. Roughly, this category consists of 
\begin{enumerate}
\item Objects: (possibly noncompact) Lagrangian submanifold whose noncompact ends live away from $\Lambda^\infty$.
\item Morphisms: wrapped Floer cohomology.
\end{enumerate}

Given a Legendrian $\Lambda^\infty$ as above, we obtain a locus in $\cotangent{Z}$ as follows.
\begin{definition} Let $\Lambda\subset \cotangent{Z}$ be given by $\Lambda=(\bR_{>0}\cdot \Lambda^\infty) \cup Z$ where $\bR_{>0}$ acts by scaling the cotangent fibers, and $Z$ is the zero section in $\cotangent{Z}$.
\end{definition}
\noindent We can also go the other way around: namely, for a given conic (i.e.~invariant under $\bR_{>0}$) Lagrangian $\Lambda$ in $\cotangent{Z} \bs Z$ we have the following.
\begin{definition} We obtain a Legendrian $\Lambda^\infty\subset \contactInfinity{Z}$ as follows. First, for any~$a$ we obtain a Legendrian in $S^*Z_a$ by $\Lambda_a=S^*Z_a\cap \Lambda$. The conicness implies that $\Lambda_a$ is preserved under the isomorphism $S^*Z_a\cong S^*Z_b$ for any~$a, b$. Hence we get a Legendrian $\Lambda^\infty$ in the abstract contact manifold~$\contactInfinity{Z}$.
\end{definition}

We may now state the following theorem.

\begin{theorem}[{Ganatra--Pardon--Shende~\cite{GPS}}]\label{theorem GPS}
There exists an equivalence
\begin{equation*}
\wSh_{\Lambda}(Z)^{\op}\simeq \cW_{\Lambda^\infty}(\cotangent{Z})
\end{equation*}
of $A_\infty$-categories, where $\op$ denotes the opposite category.
\end{theorem}

\begin{remark} We may remove the $\op$ in the above, but at the expense of negating the Liouville form, or performing a similar operation on~$\Lambda^\infty$~\cite[after Theorem~1.1]{GPS}.\end{remark}

It is expected that this theorem generalizes to the case of Weinstein manifolds instead of $\cotangent{Z}$. This is known as Kontsevich's conjecture~\cite{Kon}. Progress towards a proof has been made by many people.

A point for us is that a Landau--Ginzburg model gives an isotropic subset $\Lambda^\infty$, hence a partially wrapped Fukaya category. The A-brane category associated to a Landau--Ginzburg model is sometimes described as a Fukaya--Seidel category generated by Lefschetz thimbles (or vanishing cycles). The Fukaya--Seidel category and the partially wrapped Fukaya category associated to a single Landau--Ginzburg model are sometimes different.
\begin{example}\label{exampleFS}
We take $W(z)=z$ on $\bC^*$. Then the Fukaya--Seidel category is empty but the partially wrapped Fukaya category is equivalent to the derived category of coherent sheaves over $\bA^1$.
\end{example}
Hence, to study mirror symmetry, the partially wrapped Fukaya category is appropriate.

\smallskip
The Lagrangian skeletons for mirror Landau--Ginzburg models of toric varieties can be combinatorially defined. Such skeletons were first proposed by Fang--Liu--Treumann--Zaslow. 

\begin{notation} We use standard notation in the toric setting, as follows.

\smallskip
\begin{tabular}{p{1cm} l}
$M$ & a rank $n$ free abelian group \\
$N$ & the dual of $M$ \\
$\Sigma$ & a rational finite fan in $N_\bR:=N\otimes_\bZ\bR$ \\
$X_\Sigma$ & the toric variety of $\Sigma$ \\
$T$ & real torus $M_\bR/M$ \\
$\torusProj$ & projection $M_\bR\rightarrow T$
\end{tabular}
\medskip

\noindent For a subset $\sigma\subset N_\bR$, we set
\begin{equation*}
\sigma^\perp:=\lc m\in M_\bR\relmid n(m)=0 \text{ for any $n\in \sigma$}\rc.
\end{equation*}
\end{notation}

We then have the following.

\begin{definition}[FLTZ skeleton]
\begin{equation*}
\Lambda_{\Sigma}:=\bigcup_{\sigma\in\Sigma}\torusProj(\sigma^\perp)\times (-\sigma)\subset T\times N_\bR\cong \cotangent{T}.
\end{equation*}
\end{definition}

The following result is due to Gammage--Shende~\cite{GS}, Zhou~\cite{Z}, and Ganatra--Pardon--Shende~\cite{GPS}.
\begin{theorem}
The FLTZ skeleton at infinity $\Lambda_{\Sigma}^\infty$ is a Weinstein skeleton of a generic fiber of the Hori--Vafa mirror potential. If moreover $X_\Sigma$ is Fano, then~$\cW_{\Lambda_\Sigma^\infty}(\cotangent{T})$, and hence~$\wSh_{\Lambda_\Sigma}(T)$, is equivalent to the Fukaya--Seidel category of~$W$.
\end{theorem}
We will not define `Weinstein skeleton' here. The main point for us is that $\cW_{\Lambda_\Sigma^\infty}(\cotangent{T})$ in general can be considered as a generalization of a Fukaya--Seidel category.

Since mirror symmetry predicts an equivalence between the B-model on~$X_\Sigma$ and the mirror Landau--Ginzburg A-model, we can interpret homological mirror symmetry as an equivalence between the derived category of coherent sheaves over $X_\Sigma$ and $\wSh_\Ls(T)$. This is the content of the next subsection.

\subsection{Coherent-constructible correspondence}
In this subsection, we review the result in \cite{FLTZ} and \cite{K2} for the smooth variety case. We now take the field~$k$ to be~$\bC$. Let~$\Lambda_\Sigma$ and~$X_\Sigma$ be as in the previous section. For $\sigma\in \Sigma$, we have the corresponding affine coordinate $i_\sigma\colon U_\sigma\subset X_\Sigma$. Letting $\Qcoh$ denote the unbounded derived category of quasicoherent sheaves, there exists a unique functor
\begin{equation*}
\kappa_{\Sigma}\colon \Qcoh(X_\Sigma)\rightarrow \lSh_\Ls(X_\Sigma)
\end{equation*}
which maps $\Theta'(\sigma):=i_{\sigma*}\cO_{U_\sigma}$ to $\Theta(\sigma):=\torusProj_!\bC_{\mathrm{Int}(\sigma^\vee)}$ where $\mathrm{Int}$ is the interior and $\sigma^\vee$ is the polar dual of $\sigma$.

\begin{theorem}[\cite{K2}]\label{theorem coh-constr}
The restriction of $\kappa_\Sigma$ to $D(X_\Sigma)$ gives an equivalence
\begin{equation*}
D(X_\Sigma)\xrightarrow{\sim} \wSh_\Ls(T).
\end{equation*}
\end{theorem}
As discussed in the last part of the previous section, this equivalence is an instance of homological mirror symmetry.

\smallskip
According to \cite{K2}, we redefine the equivalence functor as
\begin{equation*}
K_\Sigma:=\kappa_\Sigma(-\otimes \omega_{\Sigma}^{-1})
\end{equation*}
where $\omega_\Sigma$ is the canonical sheaf. Since $\omega_{\Sigma}$ is invertible, we have:
\begin{corollary}[\cite{K2}]
The functor $K_\Sigma$ gives an equivalence
\begin{equation*}
D(X_\Sigma)\xrightarrow{\sim} \wSh_\Ls(T).
\end{equation*}
\end{corollary}
The reason why we use this modified functor is to get a commutativity with push-forwards on the B-side, as appeared in the proof of Theorem~\ref{theorem mirror}. 

As a corollary of the above theorem and Theorem~\ref{theorem GPS}, we get a homological mirror symmetry between coherent sheaves and a Fukaya category:
\begin{corollary}
\[D(X_{\Sigma})\xrightarrow{\sim}\cW_{\Lsinfty}(\cotangent{T})^{\mathrm{op}}\]
\end{corollary}

Next we would like to discuss the functoriality. Let $\Sigma'$ be a refinement of~$\Sigma$. We denote the corresponding map by $\pr\colon X_{\Sigma'}\rightarrow X_\Sigma$. Then there exists an inclusion relation $\Lambda_{\Sigma} \subset \Lambda_{\Sigma'}$, which induces an inclusion
\begin{equation}\label{equation big pullback}\mpr^*\colon \lSh_{\Ls}(T)\hookrightarrow \lSh_{\Lambda_{\Sigma'}}(T).\end{equation}

\begin{proposition}[Fang--Liu--Treumann--Zaslow~\cite{FLTZ}]\label{pullback}
\[\kappa_{\Sigma'}\circ \pr^*\simeq \mpr^*\circ \kappa_{\Sigma}\]
\end{proposition}

For a general inclusion of Lagrangians $\Lambda\subset \Lambda'$, the inclusion \[\lSh_{\Lambda}(T)\hookrightarrow \lSh_{\Lambda'}(T)\] induces $\cSh_{\Lambda}(T)\hookrightarrow \cSh_{\Lambda'}(T)$ but does not induce $\wSh_{\Lambda}(T)\hookrightarrow \wSh_{\Lambda'}(T)$. However, in our case, we have:

\begin{proposition}\label{pullback2} If $X_\Sigma$ is smooth, the functor~\eqref{equation big pullback} above induces a functor
\[\mpr^*\colon \wSh_{\Ls}(T)\hookrightarrow \wSh_{\Lambda_{\Sigma'}}(T).\]\end{proposition}
\begin{proof}
By Proposition~\ref{pullback} and Theorem~\ref{theorem coh-constr}, this is equivalent to saying that the pullback $\pr^* \colon \Qcoh(X_\Sigma ) \rightarrow \Qcoh(X_{\Sigma'})$ restricts to a functor between bounded derived categories, which is true under the assumption.
\end{proof}

We also note the following for later use.

\begin{proposition}\label{proposition left adjoint}
For a general inclusion of Lagrangians $\Lambda\subset \Lambda'$, a left adjoint $\iota^{\Ladj}$ of the natural embedding $\iota: \lSh_{\Lambda}(T)\hookrightarrow \lSh_{\Lambda'}(T)$ restricts to an essentially surjective functor \[\iota^w: \wSh_{\Lambda'}(T)\rightarrow \wSh_{\Lambda}(T).\] 
\end{proposition}
\begin{proof}This is by general nonsense, see Nadler~\cite[end of Section~3.6]{N}.
\end{proof}

\subsection{Toric stacks}\label{section toric stacks}
Here we recall a definition of toric stacks, following Gerashenko--Satriano~\cite{GerSat}. Let $L$ and $N$ be free abelian groups and \[f\colon L\rightarrow N\] be a morphism. We can associate a morphism between algebraic groups \[f\otimes_\bZ\bC^* \colon L\otimes_\bZ\bC^*\rightarrow N\otimes_\bZ\bC^*.\] Let $\Sigma$ be a fan in $L$ and $X_\Sigma$ be the associated toric variety. Then the toric stack associated to the data $(\Sigma,f)$ is defined by the quotient stack
\begin{equation*}
[X_\Sigma/\ker(f\otimes_\bZ\bC^*)]
\end{equation*}
where $\ker(f\otimes_\bZ\bC^*)$ acts on $X_\Sigma$ via the action of $L\otimes_\bZ\bC^*$ on $X_\Sigma$. 

\section{Schober constructions}
\label{section schober constructions}

We give the required description of the geometry of our examples, before constructing schobers on the $A$-side and $B$-side of mirror symmetry.

\subsection{Surface example}
\label{section surface geometry}
Consider the $A_1$ quotient singularity $ X_0 = \mathbb{C}^2/ \mathbb{Z}_2 $ where $\mathbb{Z}_2 = \{ \pm 1 \}.$ This has resolutions a Deligne--Mumford quotient stack, and a minimal resolution, which we denote as follows.
\[ X_- = [ \mathbb{C}^2/ \mathbb{Z}_2 ] \qquad X_+ = \widetilde{\mathbb{C}^2/ \mathbb{Z}_2} \]
The minimal resolution may be realised as
the total space of a line bundle.
\[X_+ =\operatorname{Tot} \cO_{\mathbb{P}^1}(-2)\]
The fibre product of $X_\pm$ over $X_0$ is a further Deligne--Mumford stack, as follows.
\[X_B = [\operatorname{Tot} \cO_{\mathbb{P}^1}(-1)/ \mathbb{Z}_2].\]
Here $\mathbb{Z}_2$ acts as $\pm 1$ on the fibres of the bundle. The associated morphisms \[\pr_\pm\colon X_B \to X_\pm\] are described as follows.

\begin{itemize}
\item The map $\pr_-$ is the blowup of $[ 0 / \mathbb{Z}_2 ] \subset X_-$, in other words the blowup of $0\in \mathbb{C}^2$, noting that this is equivariant with respect to the \mbox{$\mathbb{Z}_2$-actions}.
\item The map $\pr_+$ is given by the root stack construction along $\mathbb{P}^1$~(for a general discussion, see~\cite{FMN}), namely a family version of the following morphism.
\[ [\mathbb{C}/\mathbb{Z}_2 ] \to \mathbb{C} \colon z \mapsto z^2 \]
\end{itemize}

\subsubsection*{Toric description} Let $N$ be a rank~2 lattice with basis $e_1$, $e_2$, and let $M$ be its dual. The faces of the cone~$\Cone(e_1, e_1+2e_2)\subset N_\bR$ give a fan~$\Sigma_0$ representing the singularity~$X_0$. Then the small resolution $X_+$ is represented by a refinement~$\Sigma_+$ of~$\Sigma_0$ given by the faces of the following cones.
\[\Cone(e_1, e_1+e_2),\,\Cone(e_1+e_2, e_1+2e_2) \subset N_\bR \]

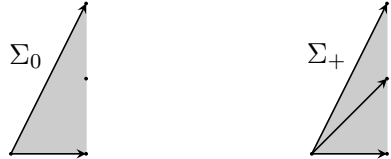
\def\gridColor{black}
\def\gridSize{0.5pt}
\[
\begin{tikzpicture}[xscale=1,scale=1]
\draw[fill,color=\fillColor] (0,0) -- (1,0) -- (1,2) -- cycle;
\draw[fill,color=\gridColor] (0,0) circle (\gridSize);
\foreach \y in {0,...,2} { \draw[fill,color=\gridColor] (1,\y) circle (\gridSize); }
\draw[semithick,-stealth] (0,0) -- (1,0);
\draw[semithick,-stealth] (0,0) -- (1,2);
\node at (0.2,1.3) {\small $\Sigma_0$};

\pgftransformxshift{4cm}
\draw[fill,color=\fillColor] (0,0) -- (1,0) -- (1,2) -- cycle;
\draw[fill,color=\gridColor] (0,0) circle (\gridSize);
\foreach \y in {0,...,2} { \draw[fill,color=\gridColor] (1,\y) circle (\gridSize); }
\draw[semithick,-stealth] (0,0) -- (1,0);
\draw[semithick,-stealth] (0,0) -- (1,1);
\draw[semithick,-stealth] (0,0) -- (1,2);
\node at (0.2,1.3) {\small $\Sigma_+$};
\end{tikzpicture}
\]

To obtain toric data for the stack $X_-$, in the sense of Section~\ref{section toric stacks}, take a further rank~$2$ lattice $L$ with basis $g_1$, $g_2$ and consider the standard fan representing~$\bC^2$, generated by faces of $\Cone(g_1, g_2)\subset L_\bR$, and denote it~$\Sigma_-$. The lattice map
\begin{align*}
f\colon & L\rightarrow N \\
& g_1 \mapsto e_1 \\
& g_2 \mapsto e_1+2e_2
\end{align*}
then induces a map between the corresponding algebraic tori
\[
 T_L \rightarrow T_N\colon
 (a,b) \mapsto (ab, b^2).
\]
The kernel of this map is $\bZ_2=\lc \pm 1 \rc$. It follows that the toric stack corresponding to the data of 
$\Sigma_L$ and $f$ is the toric stack $[\bC^2/\bZ_2]=X_-$, the stacky resolution of the $A_1$-singularity, see Figure~\ref{figure toric for Xpm}. The variety $X_+$ is realized as a toric stack by taking the data $\Sigma_+$ and~$\Idmatrix_N$.
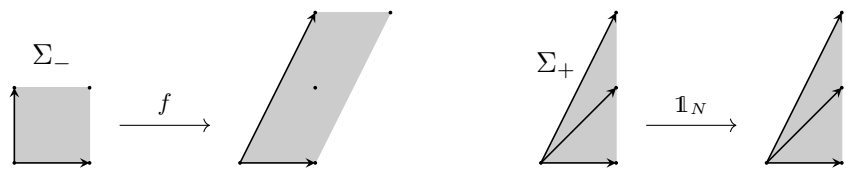
\begin{figure}[h]
\begin{tikzpicture}[xscale=1,scale=1]
\draw[fill,color=\fillColor] (0,0) -- (1,0) -- (1,1) -- (0,1) -- cycle;
\foreach \x in {0,...,1} { \foreach \y in {0,...,1} { \draw[fill,color=\gridColor] (\x,\y) circle (\gridSize); } }
\draw[semithick,-stealth] (0,0) -- (1,0);
\draw[semithick,-stealth] (0,0) -- (0,1);
\node at (0.5,1.4) {\small $\Sigma_-$};

\draw[->] (1.4,0.5) to node[above] {\scriptsize $f$} (2.6,0.5);

\pgftransformxshift{3cm}
\draw[fill,color=\fillColor] (0,0) -- (1,0) -- (2,2) -- (1,2) -- cycle;
\draw[fill,color=\gridColor] (0,0) circle (\gridSize);
\foreach \y in {0,...,2} { \draw[fill,color=\gridColor] (1,\y) circle (\gridSize); }
\draw[fill,color=\gridColor] (2,2) circle (\gridSize);
\draw[semithick,-stealth] (0,0) -- (1,0);
\draw[semithick,-stealth] (0,0) -- (1,2);

\pgftransformxshift{4cm}
\draw[fill,color=\fillColor] (0,0) -- (1,0) -- (1,2) -- cycle;
\draw[fill,color=\gridColor] (0,0) circle (\gridSize);
\foreach \y in {0,...,2} { \draw[fill,color=\gridColor] (1,\y) circle (\gridSize); }
\draw[semithick,-stealth] (0,0) -- (1,0);
\draw[semithick,-stealth] (0,0) -- (1,1);
\draw[semithick,-stealth] (0,0) -- (1,2);
\node at (0.2,1.3) {\small $\Sigma_+$};

\draw[->] (1.4,0.5) to node[above] {\scriptsize $\Idmatrix_N$} (2.6,0.5);

\pgftransformxshift{3cm}
\draw[fill,color=\fillColor] (0,0) -- (1,0) -- (1,2) -- cycle;
\draw[fill,color=\gridColor] (0,0) circle (\gridSize);
\foreach \y in {0,...,2} { \draw[fill,color=\gridColor] (1,\y) circle (\gridSize); }
\draw[semithick,-stealth] (0,0) -- (1,0);
\draw[semithick,-stealth] (0,0) -- (1,1);
\draw[semithick,-stealth] (0,0) -- (1,2);

\end{tikzpicture}
\caption{Toric data for $X_\pm$}
\label{figure toric for Xpm}
\end{figure}

For the toric data corresponding to the fibre product $X_B$, take the fan~$\Sigma_{B}$ given by faces of the following cones
\[\Cone(e_1, e_1+e_2),\,\Cone(e_1+e_2, e_2) \subset L_\bR. \]
Note that this fan refines $\Sigma_-$, and corresponds to $\operatorname{Tot} \cO_{\mathbb{P}^1}(-1)$. Then we take data $\Sigma_B$ and~$f$ to obtain~$X_B$, see Figure~\ref{figure toric for XB}. Note this is not a toric Deligne--Mumford stack in the sense of Borisov--Chen--Smith~\cite{BCS}.
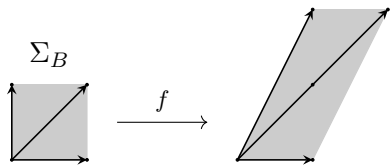
\begin{figure}[h]
\begin{tikzpicture}[xscale=1,scale=1]
\draw[fill,color=\fillColor] (0,0) -- (1,0) -- (1,1) -- (0,1) -- cycle;
\foreach \x in {0,...,1} { \foreach \y in {0,...,1} { \draw[fill,color=\gridColor] (\x,\y) circle (\gridSize); } }
\draw[semithick,-stealth] (0,0) -- (1,0);
\draw[semithick,-stealth] (0,0) -- (1,1);
\draw[semithick,-stealth] (0,0) -- (0,1);
\node at (0.5,1.4) {\small $\Sigma_B$};

\draw[->] (1.4,0.5) to node[above] {\scriptsize $f$} (2.6,0.5);

\pgftransformxshift{3cm}
\draw[fill,color=\fillColor] (0,0) -- (1,0) -- (2,2) -- (1,2) -- cycle;
\draw[fill,color=\gridColor] (0,0) circle (\gridSize);
\foreach \y in {0,...,2} { \draw[fill,color=\gridColor] (1,\y) circle (\gridSize); }
\draw[fill,color=\gridColor] (2,2) circle (\gridSize);
\draw[semithick,-stealth] (0,0) -- (1,0);
\draw[semithick,-stealth] (0,0) -- (2,2);
\draw[semithick,-stealth] (0,0) -- (1,2);
\end{tikzpicture}
\caption{Toric data for $X_B$}
\label{figure toric for XB}
\end{figure}

Now the projections $\pr_\pm$ from $X_B$ are induced by appropriate commutative squares of lattices, as shown.
\[
\begin{tikzpicture}[xscale=1,scale=1.5,baseline={([yshift=-.5ex]current bounding box.center)}]

\node (A) at (0,0) {$ X_B $};
\node (Aminus) at (0,-1) {$ X_- $};
\node (Aplus) at (0,1) {$ X_+ $};

\draw[->] (A) to node[left] {\scriptsize $\pr_-$}  (Aminus);
\draw[->] (A) to node[left] {\scriptsize $\pr_+$}  (Aplus);

\pgftransformxshift{2cm}

\node (C) at (0,0) {$ L $};
\node (Cminus) at (0,-1) {$ L $};
\node (Cplus) at (0,1) {$ N $};

\draw[-,transform canvas={xshift=+\equalsSep}] (C) to (Cminus);
\draw[-,transform canvas={xshift=-\equalsSep}] (C) to (Cminus);
\draw[->] (C) to node[left] {\scriptsize $f$}  (Cplus);

\pgftransformxshift{1cm}

\node (D) at (0,0) {$ N $};
\node (Dminus) at (0,-1) {$ N $};
\node (Dplus) at (0,1) {$ N $};

\draw[-,transform canvas={xshift=+\equalsSep}] (D) to (Dminus);
\draw[-,transform canvas={xshift=-\equalsSep}] (D) to (Dminus);
\draw[-,transform canvas={xshift=+\equalsSep}] (D) to (Dplus);
\draw[-,transform canvas={xshift=-\equalsSep}] (D) to (Dplus);

\draw[->] (Cminus) to node[above] {\scriptsize $f$}  (Dminus);
\draw[->] (C) to node[above] {\scriptsize $f$}  (D);
\draw[->] (Cplus) to node[above] {\scriptsize $\Idmatrix_N$}  (Dplus);

\end{tikzpicture}
\]

\subsection{Threefold example} 
\label{section threefold geometry}

We recall the geometry of the Atiyah flop local model. Namely, take $X_0$ to be the conifold singularity $(xy-zw=0)$ in $\mathbb{C}^4$. This has crepant resolutions $X_\pm$ with exceptional curves $E_\pm \cong \mathbb{P}^1$ such that
\[X_\pm \cong \operatorname{Tot} \cO_{E_\pm}(-1)^{\oplus 2}, \]
related by an Atiyah flop as follows.
 \[
\begin{tikzpicture}[xscale=0.9,scale=1.25,baseline={([yshift=-.5ex]current bounding box.center)}]

\node (Bplus) at (1,0) {$ X_+ $};
\node (Cminus) at (-1,0) {$ X_-\! $};
\node (Cplus) at (0,-1) {$ X_0 $};

\draw[<->,dashed] (Cminus) to (Bplus);
\draw[->] (Cminus) to (Cplus);
\draw[<-] (Cplus) to (Bplus);

\end{tikzpicture}
\]
Let $X_B$ be the fibre product of this diagram. This is isomorphic to the blowup of $X_\pm$ along $E_\pm$. Write $E$ for the common exceptional divisor $\mathbb{P}^1 \times \mathbb{P}^1$.

\subsubsection*{Toric data} 
Let $N$ be a rank $3$ lattice with basis $e_i$, and dual $M$. The set of faces of the cone $\Cone(e_1, e_2, e_1+e_3, e_2+e_3)\subset N_\bR$ gives a fan $\Sigma_0$ representing the conifold. The small resolutions $X_\pm$ come from
\begin{equation*}
\begin{split}
\Sigma_+&:=\Sigma_0\cup \{\Cone(e_2, e_1+e_3)\}\\
\Sigma_-&:=\Sigma_0\cup \{\Cone(e_1, e_2+e_3)\}
\end{split}
\end{equation*}
and the blow-up $X_B$ from
\begin{equation*}
\Sigma_B:=\Sigma_+\cup \Sigma_-\cup \{\bR_{\geq 0}\cdot (e_1+e_2+e_3)\}.
\end{equation*}
This data is sketched in Figure~\ref{figure toric atiyah}.
Note that the fan $\Sigma_B$ refines $\Sigma_\pm$ giving the morphisms \[\pr_\pm\colon X_B \rightarrow X_\pm.\]

\begin{figure}
\begin{tikzpicture}[scale=2]

\coordinate (O) at (0,0);
\coordinate (P1) at (90+20:2.2);
\coordinate (P2) at (90+15:1.5);
\coordinate (P3) at (90-7:2.2);
\coordinate (P4) at (90-25:2);
\coordinate (P3a) at (intersection cs: first line={(O) -- (P3)}, second line={(P1) -- (P4)});
\coordinate (P3b) at (intersection cs: first line={(O) -- (P3)}, second line={(P2) -- (P4)});

\coordinate (I) at (intersection cs: first line={(P4) -- (P1)}, second line={(P3) -- (P2)});
\coordinate (Ia) at (intersection cs: first line={(O) -- (I)}, second line={(P2) -- (P4)});

\fill[black!3] (O) -- (P3) -- (P1) -- cycle;
\fill[black!2] (O) -- (P3) -- (P4) -- cycle;

\fill[black!30] (O) -- (I) -- (P3) -- cycle;

\draw[semithick] (O) -- (P3);

\fill[black!15] (O) -- (I) -- (P1) -- cycle;
\fill[black!25] (O) -- (I) -- (P2) -- cycle;
\fill[black!5] (O) -- (I) -- (P4) -- cycle;

\draw[semithick,black!15] (O) -- (P3a);

\draw[semithick] (O) -- (I);

\fill[black!45] (O) -- (P2) -- (P1) -- cycle;
\fill[black!35] (O) -- (P2) -- (P4) -- cycle;

\draw[semithick,black!45] (O) -- (Ia);
\draw[semithick,black!40] (O) -- (P3b);

\draw[semithick] (O) -- (P1);
\draw[semithick] (O) -- (P2);
\draw[semithick] (O) -- (P4);

\draw[fill] (P1) circle (0.4pt);
\draw[fill] (P2) circle (0.4pt);
\draw[fill] (P3) circle (0.4pt);
\draw[fill] (P4) circle (0.4pt);

\node[anchor=south east] (C1) at (P1) {\scriptsize $e_1$};
\node[anchor=north east,transform canvas={xshift=-8pt}] (C2) at (P2) {\scriptsize $e_2$};
\node[anchor=south west] (C3) at (P3) {\scriptsize $e_1+e_3$};
\node[anchor=west] (C4) at (P4) {\scriptsize $e_2+e_3$};

\node[anchor=west] (F) at (O) {\small $\hspace{0.2cm} \Sigma_B$};

\end{tikzpicture}
\caption{Threefold toric fan in $N_{\mathbb{R}}$}
\label{figure toric atiyah}
\end{figure}
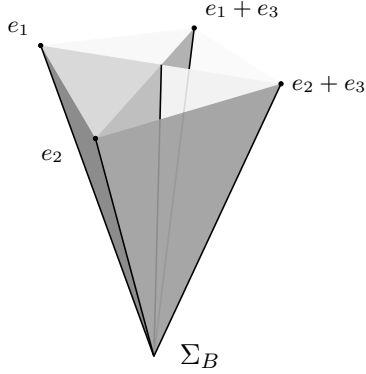

\subsection{B-side constructions}\label{section Bside schober} We construct schobers involving the categories $D(X_\pm)$ by an appropriate analysis of semiorthogonal decompositions. For the $3$-fold case which we study first, our schober should be equivalent to the one constructed, in a more general setting, by Bodzenta--Bondal in~\cite{BB}. We here use a different method to theirs, which also applies to the stacky surface case.

\subsubsection*{Threefold case} First note we have $\pr_\pm^* \colon D(X_\pm) \to D(X_B)$ with left adjoints $\pr_{\pm !}$, where we take $p_! = p_* (\omega_p[\operatorname{dim} p] \otimes \placeholder)$. By Bondal--Orlov~\cite{BO}, $\pr_{\pm *} \pr_\mp^*$ are equivalences, so by taking left adjoints $\pr_{\mp !} \pr_\pm^*$ are equivalences.
\begin{definition}Let $\window$ be full triangulated subcategory of $D(X_B)$ generated by the images of $\pr_+^*$ and $\pr_-^*$.
\end{definition}
Then we have $\pr_\pm^* \colon D(X_\pm) \to \window$ with left adjoints $\pr_{\pm !}$ obtained by restriction, because $\window$ is a full subcategory of $D(X_B)$, and we deduce that the compositions $ D(X_-) \leftrightarrow D(X_+) $ in the following diagram are equivalences~$\pr_{\mp !} \pr_\pm^*$.
\begin{equation}\label{equation B side schober}
 \begin{tikzpicture}[scale=2,baseline={([yshift=-.5ex]current bounding box.center)}]
\node (Bminus) at (-1,1) {$ D(X_-) $};
\node (Bzero) at (0,1) {$ \window $};
\node (Bplus) at (+1,1) {$ D(X_+) $};
\draw[right hook->,transform canvas={yshift=+\arrowGap}] (Bminus) to node[above] {\scriptsize $\pr_-^*$} (Bzero);
\draw[left hook->,transform canvas={yshift=+\arrowGap}] (Bplus) to node[above] {\scriptsize $\pr_+^*$} (Bzero);
\draw[<-,transform canvas={yshift=-\arrowGap}] (Bminus) to node[below] {\scriptsize $\pr_{-!}$} (Bzero);
\draw[<-,transform canvas={yshift=-\arrowGap}] (Bplus) to node[below] {\scriptsize $\pr_{+!}$} (Bzero);
\end{tikzpicture}
\end{equation}

To show that this gives (the opposite of) a spherical pair, and for use in our equivalence proof in Section~\ref{section schober equiv}, we construct semi-orthogonal decompositions of~$\window$. Set notation for the blowup of $X_\pm$ in $E_\pm$ as follows.
\begin{equation}\label{equation blowup}
 \begin{tikzpicture}[xscale=1.25,scale=0.75,baseline={([yshift=-.5ex]current bounding box.center)}]

\node (Bminus) at (-1,1) {$ X_\pm $};
\node (Bplus) at (+1,1) {$ X_B $};

\node (Cminus) at (-1,-1) {$ E_\pm $};
\node (Cplus) at (+1,-1) {$ E $};

\draw[<-] (Bminus) to node[above] {\scriptsize $\pr_\pm$}  (Bplus);
\draw[<-] (Cminus) to node[below] {\scriptsize $\excpr_\pm$}  (Cplus);

\draw[right hook->] (Cminus) to node[left] {\scriptsize $j_\pm$}  (Bminus);
\draw[right hook->] (Cplus) to node[right] {\scriptsize $i$}  (Bplus);

\end{tikzpicture}
\end{equation}

\begin{proposition}\label{proposition flober sod} The data of \eqref{equation B side schober} yields the opposite of a spherical pair. In particular, we have semi-orthogonal decompositions
\begin{equation*}
\window = \la \pr_{\pm}^* D(X_{\pm}), i_* \excpr_{\pm}^* \cO_{E_{\pm}} (-1) \ra.
\end{equation*}
\end{proposition}
\begin{proof}
Recall that by a result of Orlov~\cite{O}, see for instance \cite[Theorem~2.6]{Kuz}, we have
\[
D(X_B) = \la \pr_+^* D(X_+), i_* \excpr_+^* D(E_+) \ra.
\]
To obtain a semi-orthogonal decomposition of $\window \subset D(X_B)$, we therefore calculate the image of $\pr_-^* D(X_-) \subset \window$ under the projection to $D(E_+)$: the projection functor is~$\excpr_{+ *} i^!$. Recalling that $i^! = \omega_i [\operatorname{dim} i] \otimes i^*$, it thence suffices to calculate the image of the following functor on $D(X_-)$.
\begin{equation}\label{equation SOD projection} \excpr_{+*} \big( \omega_i \otimes i^* \pr_-^* (\placeholder) \big)
\end{equation}
Now $\omega_i = \cN_{E|X}$, and it is straightforward to show that
\begin{equation}\label{equation normal bdl} \cN_{E|X} = \excpr_+^* \cO_{E_+} (-1) \otimes \excpr_-^* \cO_{E_-} (-1),
\end{equation}
see \cite[Section~11.3]{HuybrechtsFM} for instance. The functor~\eqref{equation SOD projection} is then given as follows.
\begin{align*}
\excpr_{+*} \big( \omega_i \otimes i^* \pr_-^* (\placeholder) \big) &\cong \excpr_{+*} \big( \omega_i \otimes \excpr_-^* j_-^* (\placeholder) \big) \tag{commutativity of~\eqref{equation blowup}}\\
&\cong \excpr_{+*} \big( \excpr_+^* \cO_{E_+} (-1) \otimes \excpr_-^* \cO_{E_-} (-1) \otimes \excpr_-^* j_-^* (\placeholder) \big)\\
&\cong \cO_{E_+} (-1) \otimes \excpr_{+*} \big( \excpr_-^* \cO_{E_-} (-1) \otimes \excpr_-^* j_-^* (\placeholder) \big)\\
&\cong \cO_{E_+} (-1) \otimes \excpr_{+*} \excpr_-^* \big(  \cO_{E_-} (-1) \otimes j_-^* (\placeholder) \big)\end{align*}
Here the third line is obtained by the projection formula. Now, writing $r_\pm\colon E_\pm \to \pt$, we have $\excpr_{+*} \excpr_-^* \cong r_+^* r_{-*}$ by flat base change. The functor~\eqref{equation SOD projection} is therefore isomorphic to the composition of 
\[ r_{-*} \big(  \cO_{E_-} (-1) \otimes j_-^* (\placeholder) \big) \qquad\text{then}\qquad \cO_{E_+} (-1) \otimes r_+^* (\placeholder). \]
The first functor is essentially surjective onto $D(\pt)$: to see this, apply it to $s_-^* \cO_{E_-} (+1)$ where $s_-$ is the bundle projection of $X_-$, and note that $j_-^* s_-^* \cong \id$. The image of~\eqref{equation SOD projection} is thence the image of the second functor, namely the subcategory of $D(E_+)$ generated by $\cO_{E_+} (-1)$.
We deduce a semi-orthogonal decomposition as in the statement, with the other sign following by symmetry.

Finally we show the schober conditions. The required equivalences between $D(X_-)$ and $D(X_+)$ are explained above~\eqref{equation B side schober}. For the equivalences between the orthogonals,
these are generated by single objects, and therefore by symmetry it suffice to prove the following lemma.
\end{proof}

Writing $\cE_\pm = i_* \excpr_\pm^* \cO_{E_\pm} (-1),$ we have the following.

\begin{lemma} The image of $\cE_-$ under the projection to $D(E_+)$ is $\cE_+[-2]$. 
\end{lemma}

\begin{proof}
The projection to $D(E_+)$ is $\excpr_{+ *} i^!$. Noting that $i$ is the embedding of a divisor we have a triangle of functors
\[\id \to i^! i_* \to (\placeholder \otimes \cN_{E|X}) [-1]  \to. \]
Applying this to $\excpr_-^* \cO_{E_-} (-1)$ and using the expression \eqref{equation normal bdl} for $\cN_{E|X}$ gives
\[\excpr_-^* \cO_{E_-} (-1) \to i^! i_* \excpr_-^* \cO_{E_-} (-1) \to \excpr_+^* \cO_{E_+} (-1) \otimes \excpr_-^* \cO_{E_-} (-2) [-1] \to.\]
Applying $\excpr_{+ *} $ to the lefthand term gives zero, where we use flat base change $\excpr_{+*} \excpr_-^* \cong r_+^* r_{-*}$. Applying to the righthand term gives $\cF_+[-2]$ because $r_{-*} \cO_{E_-} (-2) =  \cO_{\pt}[-1]$, and the claim follows. 
\end{proof}

\subsubsection*{Surface case} 

We apply the same approach to the surface case, constructing the required semiorthogonal decompositions using variation of GIT. For the surfaces $X_{\pm}$, we have fibre product diagrams as in \eqref{equation blowup} as follows.
\begin{equation*}
 \begin{tikzpicture}[xscale=1.25,scale=0.75,baseline={([yshift=-.5ex]current bounding box.center)}]

\node (Bminus) at (-1,1) {$ X_\pm $};
\node (Bplus) at (+1,1) {$ X_B $};

\node (Cminus) at (-1,-1) {$ E_\pm $};
\node (Cplus) at (+1,-1) {$ E $};

\draw[<-] (Bminus) to node[above] {\scriptsize $\pr_\pm$}  (Bplus);
\draw[<-] (Cminus) to node[below] {\scriptsize $\excpr_\pm$}  (Cplus);

\draw[right hook->] (Cminus) to node[left] {\scriptsize $j_\pm$}  (Bminus);
\draw[right hook->] (Cplus) to node[right] {\scriptsize $i$}  (Bplus);

\end{tikzpicture}
\end{equation*}
Here we take 
\[ E_- = [0 / \mathbb{Z}_2]  \qquad \text{and} \qquad E_+ = \mathbb{P}^1, \]
with $j_-$ the obvious embedding, and $j_+$ the embedding of the zero section of the bundle $X_+$. The common fibre product $E$ is isomorphic to $[\mathbb{P}^1 / \mathbb{Z}_2]$ with trivial $\mathbb{Z}_2$-action.

We have $\window$ a full subcategory of $D(X_B)$ defined as in the $3$-fold case.

\begin{proposition}\label{prop schober surface} For the surfaces $X_\pm$, the data shown in~\eqref{equation B side schober} yields the opposite of a spherical pair. In particular, have semi-orthogonal decompositions
\begin{equation*}
 \window = \la \pr_{\pm}^* D(X_{\pm}), i_* \excpr_{\pm}^* \cO_{E_{\pm}} (-1) \ra.
\end{equation*}
Here $\cO_{E_-} (-1)$ is used to denote the sheaf on the stack $E_- = [0 / \mathbb{Z}_2]$ corresponding to the non-trivial irreducible representation of $\mathbb{Z}_2$.
\end{proposition}
\begin{remark} There is existing literature on semiorthogonal decompositions for `stacky blowups': this term is convenient to describe both blowups of stacks and root stack construction (see for instance~\cite[Introduction]{BLS}), as both arise naturally in birational geometry of Deligne--Mumford stacks. Derived categories of root stacks have been considered in \cite[Theorem~1.6]{IU}, and these results generalized in~\cite{BLS} and~\cite{SST}. Kawamata has also constructed semiorthogonal decompositions in this situation, for discussion see in particular \cite[Theorem~3.4]{Kaw1} with proofs found in~\cite{Kaw2,Kaw3,Kaw4}. Although the results below are likely covered by this work, we found it convenient to apply the variation of GIT method herein.
\end{remark}
\begin{proof}
We use a general construction of Coates, Iritani, Jiang, and Segal~\cite{CIJS} to realize the blowup $p_-$ as a variation of GIT. We then use standard technology which relates derived categories under variation of GIT~\cite{BFK,HL} to obtain a semi-orthogonal decomposition. The semi-orthogonal decomposition for $p_+$ follows by an appropriate adaptation of this argument.

\smallskip
\noindent {\em Case {\em ($-$)}}: First recall that $\pr_-$ is the blowup in $E_- = [ 0 / \mathbb{Z}_2 ] \subset X_-$. We choose a bundle $\cF$ on $X_-$ with a section~$\sigma$ cutting out $E_-$. Following~\cite[Section~5.2]{CIJS}, we consider the total space $ \operatorname{Tot} (\cF \oplus \cO)$ with a $\mathbb{C}^*$-action of weight~$-1$ on the fibres of $\cF$, and weight~$+1$ on the fibres of~$\cO$. Writing $(v,z)$ for fibre coordinates, and $\pi$ for the projection to $X_-$, we take the substack
\[ [M / \mathbb{C}^*] := \{ v z = \pi^*\sigma \} \subset \operatorname{Tot} (\cF \oplus \cO).\]
The stack $[M / \mathbb{C}^*]$ has GIT quotients
\[  X_- \quad\text{and}\quad \operatorname{Bl}_{E_-} X_- \cong X_B, \]
and fixed locus $M^{\mathbb{C}^*} \cong E_-$. Furthermore, we obtain a semiorthogonal decomposition as follows.
\begin{equation}\label{equation surface intermediate sod}
D(X_B) = \la \pr_-^* D(X_-), i_* \excpr_-^* D(E_-) \ra
\end{equation}

We explain some of the details, so that we can adapt them to the case of $p_+$ below. By general theory, the GIT quotients are derived equivalent to certain `windows' in $D  [M / \mathbb{C}^*]$ of the form
\[\cC_d = \big\{ \cE \in D  [M / \mathbb{C}^*] \,\big|\, \cH^\bullet Li_Z^* \cE \text{ have weights in $[0,d)$} \big\}.\]
Namely, the restriction functors 
\[ \cC_{\operatorname{rk} \cF} \to D(X_B) \qquad \cC_1 \to D(X_-) \]
are equivalences. Furthermore, there is a semiorthogonal decomposition of with $D(X_B)$ with components $D(X_-)$ and $D(E_-)$, the latter appearing with multiplicity $\operatorname{rk} \cF -1 = 1$. In particular, \cite[Lemma~5.2(1)]{CIJS} is used to show that $D(X_-)$ embeds via $\pr_-^*$, and the embedding of the orthogonal follows from~\cite[Lemma~2.3]{HLShipman}.

We then obtain the required semiorthogonal decomposition of $\window$ with component $D(X_-)$ from~\eqref{equation surface intermediate sod} by following the argument for the $3$-fold case. Note in particular that the expression \eqref{equation normal bdl} for $\cN_{E|X}$ holds verbatim.

\smallskip
\noindent {\em Case {\em ($+$)}}: We argue similarly, considering $E_+ = \mathbb{P}^1 \subset X_+$, and taking a line bundle $\cG$ on $X_+$ with a section~$\sigma$ cutting out~$E_+$. For this case we take $ \operatorname{Tot} (\cG \oplus \cO)$ with a $\mathbb{C}^*$-action of weight~$-2$ on the fibres of $\cG$, denote fibre coordinates as above, and consider the substack
\[ [N/\bC^*] := \{ v  z^2 = \pi^*\sigma \} \subset \operatorname{Tot} (\cG \oplus \cO).\]
By the arguments of~\cite[Section~5.2]{CIJS}, we see that~$[N/\bC^*]$ has a GIT quotient~$X_+$. We claim furthermore that it has $X_B$ as a GIT quotient.

For this, first note that $X_+$ may be described as a $\mathbb{C}^*$-quotient of $\mathbb{C}^3$ with weights $\left(\begin{matrix}1 & 1 & -2 \end{matrix}\right)$, coordinates~$(x_1, x_2, y)$, and semistables~$\{(x_1, x_2) \neq 0\}$. Under this description $[N/\bC^*]$ is the locus~$\{(x_1, x_2) \neq 0, \, vz^2=y\}$ in the stack~$[\mathbb{C}^5 / \mathbb{C}^{*2}]$ given as follows.
\newlength{\mycolwdD} \settowidth{\mycolwdD}{$-2$} 
\def\hspacing{14pt} \BAnewcolumntype{D}{>{\centering$}p{\mycolwdD}<{$}@{\hspace{\hspacing}}}
\[\begin{blockarray}{DDDDD}
\small{x_1}&\small{x_2}&\small{y}&\small{v}&\small{z}&\\[1ex]
\begin{block}{(DDDDD)}
1&1&-2&-2&0&\\
0&0&0&-2&1&\\
\end{block}\end{blockarray}\]
By projecting away from $y$, this is isomorphic to the left-hand stack~$[\mathbb{C}^4 / \mathbb{C}^{*2}]$ below. By row and column operations, this is isomorphic to the right-hand one, which can be seen to have a GIT quotient $X_B= [\operatorname{Tot} \cO_{\mathbb{P}^1}(-1)/ \mathbb{Z}_2]$. By inspection, the variation of GIT for $[N/\bC^*]$ induces the map $p_+\colon X_B \to X_+$.
\[\begin{blockarray}{DDDD}
&\small{x_1}&\small{x_2}&\small{v}&\small{z}&\\[1ex]
\begin{block}{(DDDD)}
&1 & 1 & -2 & 0&\\
& 0 & 0 & -2 & 1&\\
\end{block}\end{blockarray}
\qquad\longrightarrow\qquad\quad
\begin{blockarray}{DDDD}
&\small{x_1}&\small{x_2}&\small{z}&\small{v}&\\[1ex]
\begin{block}{(DDDD)}
&1 & 1 & -1 & 0 &\\
& 0 & 0 & 1 & -2 &\\
\end{block}\end{blockarray}
\]

We then again use the argument of~\cite[Section~5.2]{CIJS} to obtain a semi-orthogonal decomposition of $\window$ with component $D(X_+)$. This proceeds as above, except that $\cC_{2 \cdot\operatorname{rk} \cG}$ takes the place of $\cC_{\operatorname{rk} \cF}$. In particular, the multiplicity of components $D(E_+)$ in the semiorthogonal decomposition is now given by $2 \cdot \operatorname{rk} \cG -1 = 2\cdot 1 - 1 =1,$ as required.

\smallskip
Finally we show the schober conditions. The required equivalences between $D(X_-)$ and~$D(X_+)$ are an instance of the main theorem of Bridgeland--King--Reid~\cite{BKR}. The proof of the equivalences between orthogonals goes through as for the $3$-fold case, using the expression \eqref{equation normal bdl} for $\cN_{E|X}$ as before, which applies verbatim.
\end{proof}

 \subsection{A-side constructions}\label{section schober Aside} We  set $\Lambda':=\Lambda_{\Sigma'}$ and $\Lambda_\pm:=\Lambda_{\Sigma_\pm}$. We also write $\LambdaW :=\Lambda_+\cup \Lambda_-$.

\begin{remark} The skeleton $\LambdaW$ is not represented as $\Lambda_\Sigma$ by some $\Sigma$.\end{remark}

Using that $\Lambda_\pm \subset \LambdaW$, there exist natural embeddings:
\begin{equation}\label{equation big schober embedding}
 \lSh_{\Lambda_\pm}(T)\hookrightarrow \lSh_{\LambdaW}(T)\end{equation}
We want to define such a functor on the level of wrapped constructible sheaf categories. For this, recall the following.

\begin{enumerate}\label{wrapped factors}
\item\label{wrapped factors 1} Using smoothness of $X_{\pm}$ and~$X_{B}$, Proposition~\ref{pullback2} gives functors
\begin{equation*}\label{equation wrapped big embedding}\wSh_{\Lambda_\pm}(T)\hookrightarrow \wSh_{\Lambda_B}(T).\end{equation*}
\item\label{wrapped factors 2} Using the inclusion $\LambdaW \subset \Lambda_B$, Proposition~\ref{proposition left adjoint} gives a functor
\[ \iota^w \colon \wSh_{\Lambda_B}(T) \rightarrow \wSh_{\LambdaWsmall}(T). \]
\end{enumerate}

\begin{definition} Composing~(\ref{wrapped factors 1}) and~(\ref{wrapped factors 2}), we write 
\begin{equation*}\label{equation wrapped schober embedding}\mpr^*_\pm\colon \wSh_{\Lambda_\pm}(T)\rightarrow \wSh_{\LambdaWsmall}(T).\end{equation*}
Using Proposition~\ref{proposition left adjoint}, we also have functors
 \[\mpr_{\pm !}\colon \wSh_{\LambdaWsmall}(T) \to \wSh_{\Lambda_\pm}(T).\]
\end{definition}

\begin{remark}We use the notations $\mpr^*_\pm$, $\mpr_{\pm !}$ because these functors will turn out to be mirror to~$\pr^*_\pm$, $\pr_{\pm !}$ between $D(X_\pm)$ and~$\cP_0$.
\end{remark}

\begin{proposition}\label{proposition embedding reexpressed} The functor~$\mpr^*_\pm$ is isomorphic to the restriction of functor~\eqref{equation big schober embedding} to $\wSh_{\Lambda_\pm}(T)$. In particular, it is an embedding.\end{proposition}

\begin{proof}By functoriality, composing~\eqref{equation big schober embedding} with 
the embedding~\[\iota \colon \lSh_{\LambdaWsmall}(T)\hookrightarrow \lSh_{\Lambda_B}(T)\] gives the embeddings~$\iota_\pm\colon \lSh_{\Lambda_\pm}(T)\hookrightarrow \lSh_{\Lambda_B}(T)$. Writing~$\iota^{\Ladj}$ for a left adjoint, we have $\iota^{\Ladj} \compose \iota \simeq \id$. Therefore~\eqref{equation big schober embedding} is isomorphic to the composition of $\iota_\pm$ with~$\iota^{\Ladj}$. Restricting the latter composition to $\wSh_{\Lambda_\pm}(T)$ gives~$\mpr^*_\pm$ by definition.\end{proof}

\begin{claim} Embeddings $\mpr_\pm^*$ give a spherical pair,  which is mirror to the B-side spherical pair from Section~\ref{section Bside schober}.\end{claim}
We will prove this claim by using mirror symmetry in the next section.

   \section{Mirror equivalences}
\label{section schober equiv}
In this section we prove homological mirror symmetry statements for the A-side and B-side schobers of the previous section, obtaining proofs of Theorems~\ref{theorem schober equivalence} and~\ref{keytheorem flober equiv}.

\subsection{Threefold proof}
Theorem~\ref{theorem coh-constr} gives equivalences
\[
{\kappa_{\Sigma_B}} \colon D(X_B)\xrightarrow{\sim}\wSh_{\Lambda_B}(T) \qquad
{\kappa_{\Sigma_\pm}} \colon D(X_\pm)\xrightarrow{\sim}\wSh_{\Lambda_\pm}(T)
\]
which we write as ${\kappa_{B}}$ and ${\kappa_{\pm}}$ for brevity.

\subsubsection*{Mirror symmetry at the conifold point}
Let us consider the composition
\begin{equation*}
\kappa_\windowLetter\colon \window \hookrightarrow D(X_B)\xrightarrow{\kappa_{B}}\wSh_{\Lambda_B}(T)\xrightarrow{\iota^w}\wSh_{\LambdaWsmall}(T). 
\end{equation*}

\begin{claim}
$\kappa_\windowLetter$ is fully faithful.
\end{claim}
\begin{proof}
Let us consider the precomposition of the inclusion $D(X_{\pm})\hookrightarrow\window$ with $\kappa_\windowLetter$ and denote it $\kappa_{\windowLetter\pm}$. Then by the functoriality in Proposition~\ref{pullback}, this is the same as the composition of $\kappa_{\pm}$ and the functor
$\mpr_\pm^*\colon\wSh_{\Lambda_\pm}(T)\hookrightarrow \wSh_{\LambdaWsmall}(T)$ constructed in Section~\ref{section schober Aside}. Hence $\kappa_{\windowLetter\pm}$ is fully faithful. Since $D(X_{\pm})$ jointly generate $\window$ by definition, the functor $\kappa_\windowLetter$ is also fully faithful.
\end{proof}

Take a point $(x,\xi)\in \Omega_T$ which lies in $\LambdaW$, but not in $\Lambda_\flopSide$. Let $\microsky_{x, \xi}$ be a microlocal skyscraper sheaf.

\begin{remark}
Skyscraper sheaves are exceptional objects. However microlocal skyscraper sheaves are not exceptional in general. 
\end{remark}
However, in our case,
\begin{claim}
$\microsky_{x,\xi}$ is exceptional.
\end{claim}
\begin{proof}
Let us consider the locally closed subset $D$ in $M_\bR$, as shown in Figure~\ref{figure region}, consisting of those $m\in M_\bR$ such that the following hold.
\begin{equation*}
 \la m, e_1+e_3\ra \leq -1 \quad \la m, e_2\ra\geq 0 \quad \la m, e_1\ra >-1 \quad \la m, -e_2+e_3\ra >-1
\end{equation*}
The $1$-cell~$\Delta$ of~$D$ cut out by
\begin{equation*}
 \la m, e_1+e_3\ra = -1 \quad \la m, e_2\ra = 0
\end{equation*}
has conormal~$\LambdaW-\Lambda_\flopSide$, and it follows that the sheaf $\torusProj_!\bC_D$ gives $\microsky_{x,\xi}$ by non-characteristic deformation as in \cite{K1}. Since $\RHom(\torusProj_!\bC_D, \torusProj_!\bC_D)$ is the microstalk of $\torusProj_!\bC_D$, which is rank~$1$ and degree~$0$ from the picture below, the claim follows.
\end{proof}

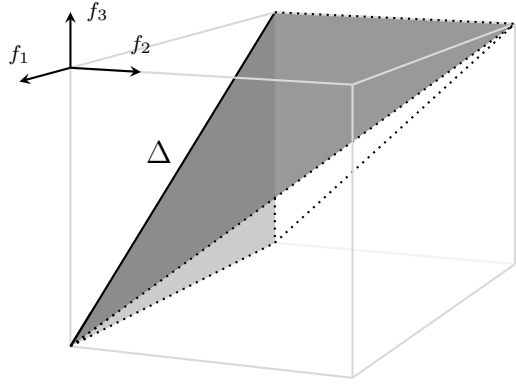
\begin{figure}[h]
  \begin{center}
\begin{tikzpicture}[scale=3]

	\def\lightGray{gray!10}
	\def\midGray{gray!30}
	\def\labelShift{0.12cm}
	
	\coordinate (P1) at (-25cm,1.5cm); 
	\coordinate (P2) at (4cm,1.5cm); 

	\coordinate (A1) at (0em,0cm); 
	\coordinate (A2) at (0em,-1.3cm); 

	\coordinate (A3) at ($(P1)!.95!(A2)$); 
	\coordinate (A4) at ($(P1)!.95!(A1)$); 
	\coordinate (A7) at ($(P2)!.82!(A2)$); 
	\coordinate (A8) at ($(P2)!.82!(A1)$); 
	\coordinate (A5) at
	  (intersection cs: first line={(A8) -- (P1)},
			    second line={(A4) -- (P2)}); 
	\coordinate (A6) at
	  (intersection cs: first line={(A7) -- (P1)}, 
			    second line={(A3) -- (P2)}); 
			   
	\coordinate (Y1) at ($(A4)!-.25!(A5)$);
	\coordinate (Y2) at ($(A4)!.25!(A1)$);
	\coordinate (Y3) at ($(A4)!-.2!(A3)$);
	
	\coordinate (L1) at ($(Y1)+(0,\labelShift)$);
	\coordinate (L2) at ($(Y2)+(0,\labelShift)$);
	\coordinate (L3) at ($(Y3)+(\labelShift,0)$);

	\draw[thick,\lightGray] (A3) -- (A6);
	\draw[thick,\lightGray] (A5) -- (A6);
	\draw[thick,\lightGray] (A7) -- (A6);

	\fill[gray!40] (A5) -- (A3) -- (A6) -- cycle;
	\fill[gray!80] (A5) -- (A3) -- (A8) -- cycle;

	\coordinate (Amid) at
	  (intersection cs: first line={(A5) -- (A6)},
			    second line={(A3) -- (A8)});
	\fill[gray!90] (A5) -- (A3) -- (Amid) -- cycle;

	\draw[thick,dotted] (A3) -- (A6);
	\draw[thick,dotted] (A6) -- (A8);
	\draw[thick] (A5) -- (A3); 
	\draw[thick,dotted] (A3) -- (A8);
	\draw[thick,dotted] (Amid) -- (A6);

	\draw[thick,\midGray] (A1) -- (A2);
	\draw[thick,\midGray] (A3) -- (A4);
	\draw[thick,\midGray] (A7) -- (A8);
	\draw[thick,\midGray] (A1) -- (A4);
	\draw[thick,\midGray] (A1) -- (A8);
	\draw[thick,\midGray] (A2) -- (A3);
	\draw[thick,\midGray] (A2) -- (A7);
	\draw[thick,\midGray] (A4) -- (A5);
	\draw[thick,\midGray] (A8) -- (A5);
	
	\draw[thick,dotted] (A5) -- (A8);
	\draw[thick,dotted] ($(A3)!.99!(A8)$) -- (A8);
	\draw[thick,dotted] (A3) -- ($(A3)!.1!(A8)$);
	\draw[thick] (A3) -- ($(A3)!.1!(A5)$);
	\draw[thick,dotted] (A3) -- ($(A3)!.1!(A6)$);

	\draw[thick,-stealth] (A4) -- (Y1); \node at (L1) {$\scriptstyle f_1$};
	\draw[thick,-stealth] (A4) -- (Y2); \node at (L2) {$\scriptstyle f_2$};
	\draw[thick,-stealth] (A4) -- (Y3); \node at (L3) {$\scriptstyle f_3$};

	\node (D1) at ($(A3)!.5!(A5)+(-0.5*\labelShift,\labelShift)$) {$\Delta$};	

\end{tikzpicture}
    \caption{The region $D$ in $M_\bR$: the basis~$\{ f_i \}$ is dual to~$\{ e_i \}$, and shaded faces are included in~$D$.}
    \label{figure region}
  \end{center}
\end{figure}

\begin{proposition}\label{sod}
\[\wSh_{\LambdaWsmall}(T)\simeq \la \wSh_{\Lambda_\flopSide}(T), \microsky_{x,\xi}\ra.\]
\end{proposition}
\begin{proof}
Since $\microsky_{x,\xi}$ is exceptional, we have a semi-orthogonal decomposition
\[\wSh_{\LambdaWsmall}(T)\simeq \la \cC, \microsky_{x,\xi}\ra\]
where $\cC$ is the left orthogonal of $\microsky_{x, \xi}$. By the definition of microlocal skyscraper, $\cC\subset \lSh_{\Lambda_-}(T)$.

Since we now have an inclusion $\wSh_{\Lambda_\flopSide}(T)\subset \wSh_{\LambdaWsmall}(T)$, we also have $j\colon \wSh_{\Lambda_\flopSide}(T)\hookrightarrow \cC$. So it suffices to show that this latter inclusion is essentially surjective.

By the definition of semi-orthogonal decomposition, we have a left adjoint $i^{\Ladj}\colon \wSh_{\LambdaWsmall}(T)\rightarrow \cC$ of $i\colon \cC\hookrightarrow \wSh_{\LambdaWsmall}(T)$. On the other hand, we also have a composition of the inclusion $j\colon \wSh_{\Lambda_\flopSide}(T)\hookrightarrow \cC$ and \[\iota^w\colon \wSh_{\LambdaWsmall}(T)\rightarrow \wSh_{\Lambda_\flopSide}(T)\] which is the left adjoint of $\iota\colon \wSh_{\Lambda_\flopSide}(T)\hookrightarrow \wSh_{\LambdaWsmall}(T)$. Then we have, for any $\cE\in\wSh_{\LambdaWsmall}(T)$ and $\cF\in \cC$,
\[
\begin{split}
\Hom_{\cC}(j\circ \iota^w(\cE), \cF)&\simeq \Hom_{\lSh_{\Lambda_\flopSide}}(\iota^w(\cE), \cF)\\
&\simeq \Hom_{\lSh_{\LambdaWsmall}}(\cE, \cF)\\
&\simeq \Hom_{\wSh_{\LambdaWsmall}}(\cE, i(\cF))\\
&\simeq \Hom_{\cC}(i^L(\cE),\cF).
\end{split}
\]
Hence we have $i^{\Ladj}\simeq j\circ \iota^w$. Since $i^{\Ladj}$ is essentially surjective, we can conclude that $j$ is also essentially surjective. This completes the proof.
\end{proof}

\begin{remark}
We also have a semi-orthogonal decomposition of the form
\[\wSh_{\LambdaWsmall}(T)\simeq \la \wSh_{\Lambda_+}(T), \microsky_{x,\xi}\ra\]
where $(x, \xi)\in \bigcup\Lambda_\pm\bs \Lambda_-$. In this case, $F_{x, \xi}$ is $\pi_!\bC_{D'}$ where $D'$ is the image of $D$ under the reflection $e_1\leftrightarrow e_2$.
\end{remark}

We also have a semiorthogonal decomposition $\window \simeq \la D(X_\flopSide), D(\pt)\ra$ from Proposition~\ref{proposition flober sod}. Note that the functor $\kappa_B$ restricts to the equivalence $\kappa_-$.

\begin{proposition}\label{conifoldequiv}
$\kappa_\windowLetter$ is an equivalence.
\end{proposition}
\begin{proof}
We have already shown that $\kappa_\windowLetter$ is fully faithful, and that its restricts on $D(X_\flopSide)$ to an equivalence $D(X_\flopSide)\rightarrow \wSh_{\Lambda_\flopSide}(T)$. Using the semiorthogonal decompositions above we deduce that the restriction of $\kappa_\windowLetter$ to the orthogonal $D(\pt)$ of $D(X_\flopSide)$ gives a functor  
\[\kappa_{\pt} \colon D(\pt) \rightarrow \la \microsky_{x,\xi}\ra.\]
This is fully faithful, therefore via the equivalence $\la \microsky_{x,\xi}\ra\simeq D(\pt)$ it is isomorphic to a homological shift $[s]$. Hence $\kappa_{\pt}$ is essentially surjective, and so the same holds for $\kappa_\windowLetter$, completing the proof.
\end{proof}

\subsubsection*{Relating mirror schobers} We use Proposition~\ref{conifoldequiv} above to conclude equivalences of schobers.

We now have an opposite spherical pair
\[
 \begin{tikzpicture}[scale=2]
\node (Bminus) at (-1.3,1) {$ D(X_-) $};
\node (Bzero) at (0,1) {$ \window $};
\node (Bplus) at (+1.3,1) {$ D(X_+) $};
\draw[right hook->,transform canvas={yshift=+\arrowGap}] (Bminus) to node[above] {\scriptsize $\pr_-^*$} (Bzero);
\draw[left hook->,transform canvas={yshift=+\arrowGap}] (Bplus) to node[above] {\scriptsize $\pr_+^*$} (Bzero);
\draw[<-,transform canvas={yshift=-\arrowGap}] (Bminus) to node[below] {\scriptsize $\pr_{-!}$} (Bzero);
\draw[<-,transform canvas={yshift=-\arrowGap}] (Bplus) to node[below] {\scriptsize $\pr_{+!}$} (Bzero);
\end{tikzpicture}
\]
Let us set $\cP_0':=\cP_0\otimes \omega_{X_B}\subset D(X_B)$.
By tensoring with the canonical sheaves on each category of the above schober, we get an equivalent opposite spherical pair.
\[
 \begin{tikzpicture}[scale=2]
\node (Bminus) at (-1.3,1) {$ D(X_-) $};
\node (Bzero) at (0,1) {$ \window'$};
\node (Bplus) at (+1.3,1) {$ D(X_+) $};
\draw[right hook->,transform canvas={yshift=+\arrowGap}] (Bminus) to node[above] {\scriptsize $\pr_-^!$} (Bzero);
\draw[left hook->,transform canvas={yshift=+\arrowGap}] (Bplus) to node[above] {\scriptsize $\pr_+^!$} (Bzero);
\draw[<-,transform canvas={yshift=-\arrowGap}] (Bminus) to node[below] {\scriptsize $\pr_{-*}$} (Bzero);
\draw[<-,transform canvas={yshift=-\arrowGap}] (Bplus) to node[below] {\scriptsize $\pr_{+*}$} (Bzero);
\end{tikzpicture}
\]
To show this we use the isomorphism
\[p_\pm^! \cong \omega_{X_B}\otimes p_\pm^*\big(\omega_{X_\pm}^{-1} \otimes -\big),\]
which holds because the relative dimension of $p_\pm$ is zero, and its left adjoint.

\begin{theorem}\label{theorem mirror} The opposite spherical pair
\[
 \begin{tikzpicture}[scale=2]
\node (Bminus) at (-1.3,1) {$ D(X_-) $};
\node (Bzero) at (0,1) {$ \window' $};
\node (Bplus) at (+1.3,1) {$ D(X_+) $};
\draw[right hook->,transform canvas={yshift=+\arrowGap}] (Bminus) to node[above] {\scriptsize $\pr_-^!$} (Bzero);
\draw[left hook->,transform canvas={yshift=+\arrowGap}] (Bplus) to node[above] {\scriptsize $\pr_+^!$} (Bzero);
\draw[<-,transform canvas={yshift=-\arrowGap}] (Bminus) to node[below] {\scriptsize $\pr_{-*}$} (Bzero);
\draw[<-,transform canvas={yshift=-\arrowGap}] (Bplus) to node[below] {\scriptsize $\pr_{+*}$} (Bzero);
\end{tikzpicture}
\]
is equivalent via the mirror symmetry functor $K_{\Sigma_{\pm}}$ to an opposite spherical pair as follows.
\[
 \begin{tikzpicture}[scale=2.2]
\node (Bminus) at (-1.5,1) {$ \wSh_{\Lambda_-}(T) $};
\node (Bzero) at (0,1) {$ \wSh_{\LambdaWsmall}(T) $};
\node (Bplus) at (+1.5,1) {$ \wSh_{\Lambda_+}(T) $};
\draw[right hook->,transform canvas={yshift=+\arrowGap}] (Bminus) to node[above] {\scriptsize $\mpr_{-}^*$} (Bzero);
\draw[left hook->,transform canvas={yshift=+\arrowGap}] (Bplus) to node[above] {\scriptsize $\mpr_{+}^*$} (Bzero);
\draw[<-,transform canvas={yshift=-\arrowGap}] (Bminus) to node[below] {\scriptsize $\mpr_{-!}$} (Bzero);
\draw[<-,transform canvas={yshift=-\arrowGap}] (Bplus) to node[below] {\scriptsize $\mpr_{-!}$} (Bzero);
\end{tikzpicture}
\]
\end{theorem}

\begin{proof} 
By the argument above, in particular using Theorem~\ref{theorem coh-constr}, Proposition~\ref{pullback}, and Proposition \ref{conifoldequiv}, we have commutative squares as follows
\[
 \begin{tikzpicture}[xscale=1.5,scale=2]

\node (Bminus) at (-1,1) {$ D(X_-) $};
\node (Bzero) at (0,1) {$ \cP_0 $};
\node (Bplus) at (+1,1) {$ D(X_+) $};

\node (Aminus) at (-1,0) {$ \wSh_{\Lambda_-}(T) $};
\node (Azero) at (0,0) {$ \wSh_{\LambdaWsmall}(T) $};
\node (Aplus) at (+1,0) {$ \wSh_{\Lambda_+}(T) $};

\draw[->] (Bminus) to node[above,sloped] {\scriptsize $\sim$} (Aminus);
\draw[->] (Bzero) to node[above,sloped] {\scriptsize $\sim$} (Azero);
\draw[->] (Bplus) to node[above,sloped] {\scriptsize $\sim$} (Aplus);

\draw[->] (Bminus) to node[above] {\scriptsize $\pr_{-}^*$} (Bzero);
\draw[->] (Bplus) to node[above] {\scriptsize $\pr_{+}^*$} (Bzero);

\draw[->] (Aminus) to node[above] {\scriptsize $\mpr_{-}^*$} (Azero);
\draw[->] (Aplus) to node[above] {\scriptsize $\mpr_{+}^*$} (Azero);

\end{tikzpicture}
\]
where the vertical arrows are the functors $\kappa$.

Taking left adjoints, we furthermore obtain the following diagram.
\[
 \begin{tikzpicture}[xscale=1.5,scale=2]

\node (Bminus) at (-1,1) {$ D(X_-) $};
\node (Bzero) at (0,1) {$ \cP_0 $};
\node (Bplus) at (+1,1) {$ D(X_+) $};

\node (Aminus) at (-1,0) {$ \wSh_{\Lambda_-}(T) $};
\node (Azero) at (0,0) {$ \wSh_{\LambdaWsmall}(T) $};
\node (Aplus) at (+1,0) {$ \wSh_{\Lambda_+}(T) $};

\draw[<-] (Bminus) to node[above,sloped] {\scriptsize $\sim$} (Aminus);
\draw[<-] (Bzero) to node[above,sloped] {\scriptsize $\sim$} (Azero);
\draw[<-] (Bplus) to node[above,sloped] {\scriptsize $\sim$} (Aplus);

\draw[<-] (Bminus) to node[above] {\scriptsize $\pr_{-!}$} (Bzero);
\draw[<-] (Bplus) to node[above] {\scriptsize $\pr_{+!}$} (Bzero);

\draw[<-] (Aminus) to node[above] {\scriptsize $\mpr_{-!}$} (Azero);
\draw[<-] (Aplus) to node[above] {\scriptsize $\mpr_{+!}$} (Azero);

\end{tikzpicture}
\]
Since the upper lines of the diagrams give the data of a weak spherical pair and satisfy the condition to be a spherical pair, we deduce that the lower lines give the data of an equivalent spherical pair. Note that the associated semi-orthogonal decompositions are given in Proposition~\ref{sod} and Remark after the proposition.

Composing the tensor of the inverse of canonical sheaves with the above schober equivalence, we have
\[
 \begin{tikzpicture}[xscale=1.5,scale=2]

\node (Bminus) at (-1,1) {$ D(X_-) $};
\node (Bzero) at (0,1) {$ \cP_0' $};
\node (Bplus) at (+1,1) {$ D(X_+) $};

\node (Aminus) at (-1,0) {$ \wSh_{\Lambda_-}(T) $};
\node (Azero) at (0,0) {$ \wSh_{\LambdaWsmall}(T) $};
\node (Aplus) at (+1,0) {$ \wSh_{\Lambda_+}(T) $};

\draw[->] (Bminus) to node[above,sloped] {\scriptsize $\sim$} (Aminus);
\draw[->] (Bzero) to node[above,sloped] {\scriptsize $\sim$} (Azero);
\draw[->] (Bplus) to node[above,sloped] {\scriptsize $\sim$} (Aplus);

\draw[->] (Bminus) to node[above] {\scriptsize $\pr_{-}^!$} (Bzero);
\draw[->] (Bplus) to node[above] {\scriptsize $\pr_{+}^!$} (Bzero);

\draw[->] (Aminus) to node[above] {\scriptsize $\mpr_{-}^*$} (Azero);
\draw[->] (Aplus) to node[above] {\scriptsize $\mpr_{+}^*$} (Azero);

\end{tikzpicture}
\]
where the vertical arrows are $\kappa\circ (-\otimes \omega^{-1})=K$.
\end{proof}

Let $\bD:=\cHom(-, \omega)$ be the Grothendieck duality. Then
\begin{equation*}
\begin{split}
\bD\circ p^!\circ \bD&\simeq\bD\circ \big(p^*\cHom(-, \omega)\otimes \omega \otimes p^*\omega^{-1}\big)\\
&\simeq\bD\circ \big(p^*(-)^\vee\otimes p^*\omega\otimes \omega \otimes p^*\omega^{-1}\big)\\
&\simeq\bD\circ \big(p^*(-)^\vee\otimes \omega\big)\\
&\simeq\cHom\big(p^*(-)^\vee\otimes \omega, \omega\big)\\
&\simeq\big(p^*(-)^\vee\big)^\vee\\
&\simeq p^*.
\end{split}
\end{equation*}
Therefore applying $\bD$ to the opposite spherical pair given by the top line in the diagram below
\[
 \begin{tikzpicture}[scale=2, xscale=1.2]

\node (Bminus) at (-1,1) {$ D(X_-) $};
\node (Bzero) at (0,1) {$ \cP_0' $};
\node (Bplus) at (+1,1) {$ D(X_+) $};

\node (Aminus) at (-1,0) {$ D(X_-) $};
\node (Azero) at (0,0) {$ \bD\cP_0' $};
\node (Aplus) at (+1,0) {$ D(X_+) $};

\draw[->] (Bminus) to node[left] {$\bD$} (Aminus);
\draw[->] (Bzero) to node[right] {$\bD$} (Azero);
\draw[->] (Bplus) to node[right] {$\bD$} (Aplus);

\draw[->] (Bminus) to node[above] {\scriptsize $\pr_{-}^!$} (Bzero);
\draw[->] (Bplus) to node[above] {\scriptsize $\pr_{+}^!$} (Bzero);

\draw[->] (Aminus) to node[above] {\scriptsize $p_{-}^*$} (Azero);
\draw[->] (Aplus) to node[above] {\scriptsize $p_{+}^*$} (Azero);

\end{tikzpicture}
\]
we get an anti-equivalent spherical pair. By the construction, $\bD\cP_0'$ is the image of the functors~$p_\pm^*$ and hence~$\bD\cP_0'=\cP_0$. Consequently, we have a spherical pair as follows.
\[
 \begin{tikzpicture}[scale=2]
\node (Bminus) at (-1.3,1) {$ D(X_-) $};
\node (Bzero) at (0,1) {$ \window $};
\node (Bplus) at (+1.3,1) {$ D(X_+) $};
\draw[right hook->,transform canvas={yshift=+\arrowGap}] (Bminus) to node[above] {\scriptsize $\pr_-^*$} (Bzero);
\draw[left hook->,transform canvas={yshift=+\arrowGap}] (Bplus) to node[above] {\scriptsize $\pr_+^*$} (Bzero);
\draw[<-,transform canvas={yshift=-\arrowGap}] (Bminus) to node[below] {\scriptsize $\pr_{-*}$} (Bzero);
\draw[<-,transform canvas={yshift=-\arrowGap}] (Bplus) to node[below] {\scriptsize $\pr_{+*}$} (Bzero);
\end{tikzpicture}
\]
Combining this with Theorem~\ref{theorem GPS} we obtain the following.

\begin{theorem}\label{theorem mirror Fukaya} In the $3$-fold setting of Section~\ref{section threefold geometry}, there exists a spherical pair as follows
\[
 \begin{tikzpicture}[scale=2,xscale=1.3]
\node (Bminus) at (-1.5,1) {$ \cW_{\Lambda_-^\infty}(\cotangent{T}) $};
\node (Bzero) at (0,1) {$ \cW_{\LambdaWsmall^\infty}(\cotangent{T}) $};
\node (Bplus) at (+1.5,1) {$ \cW_{\Lambda_+^\infty}(\cotangent{T}) $};
\draw[right hook->,transform canvas={yshift=+\arrowGap}] (Bminus) to node[above] {\scriptsize $\mpr_{-}^*$} (Bzero);
\draw[left hook->,transform canvas={yshift=+\arrowGap}] (Bplus) to node[above] {\scriptsize $\mpr_{+}^*$} (Bzero);
\draw[<-,transform canvas={yshift=-\arrowGap}] (Bminus) to node[below] {\scriptsize $\mpr_{-*}$} (Bzero);
\draw[<-,transform canvas={yshift=-\arrowGap}] (Bplus) to node[below] {\scriptsize $\mpr_{-*}$} (Bzero);
\end{tikzpicture}
\]
which is equivalent via mirror symmetry to the spherical pair as follows.
\[
 \begin{tikzpicture}[scale=2]
\node (Bminus) at (-1.3,1) {$ D(X_-) $};
\node (Bzero) at (0,1) {$ \window $};
\node (Bplus) at (+1.3,1) {$ D(X_+) $};
\draw[right hook->,transform canvas={yshift=+\arrowGap}] (Bminus) to node[above] {\scriptsize $\pr_-^*$} (Bzero);
\draw[left hook->,transform canvas={yshift=+\arrowGap}] (Bplus) to node[above] {\scriptsize $\pr_+^*$} (Bzero);
\draw[<-,transform canvas={yshift=-\arrowGap}] (Bminus) to node[below] {\scriptsize $\pr_{-*}$} (Bzero);
\draw[<-,transform canvas={yshift=-\arrowGap}] (Bplus) to node[below] {\scriptsize $\pr_{+*}$} (Bzero);
\end{tikzpicture}
\]
\end{theorem}
\begin{proof}
Note that $\LambdaWsmall^\infty=(\LambdaWsmall)^{\infty}$ by the definition.
We take the functors between Fukaya categories as the compositions of the functor in Theorem~\ref{theorem GPS} with the functors between wrapped constructible sheaves. The commutativity and equivalences follow from this description.\end{proof}

\begin{remark}
The functors $q_{\pm*}$ between Fukaya categories are isomorphic to stop removal functors as given in~\cite{GPS}. The functors $q^{*}_\pm$ are a bit more subtle, as we used the smoothness of mirrors to construct them. A Fukaya-categorical description may be as follows: the image of a Lagrangian submanifold is its result under the Reeb flow until it stops. However, making this claim precise is a problem.
\end{remark}

\subsection{Surface proof}
\label{section surface schober equiv}

We return to the setting of Section~\ref{section surface geometry}, in particular we use notation as follows.
\[ X_- = [ \mathbb{C}^2/ \mathbb{Z}_2 ] \qquad X_+ = \widetilde{\mathbb{C}^2/ \mathbb{Z}_2} \]

For a toric Deligne--Mumford stack, we have to generalize the definition of $\Lambda_\Sigma$. We only describe the result:
\begin{equation*}
\Lambda_{-}:=\Lambda_{\Sigma_0}\cup (\Lambda_{\Sigma_0}+\tfrac{1}{2}\cdot [e_1^\vee])
\end{equation*}
where $+\frac{1}{2}\cdot [e_1^\vee]$ means the translation by the class $\frac{1}{2}\cdot [e_1^\vee]\in T=M_\bR/M$.

As before, let us take a point $(x,\xi)\in \cotangent{T}$ which lies in $\LambdaW$, but not in~$\Lambda_\flopSide$. Let $\microsky_{x, \xi}$ be a microlocal skyscraper sheaf.

\begin{lemma} $\microsky_{x, \xi}$ is exceptional.
\begin{proof} Again, we have an explicit description.
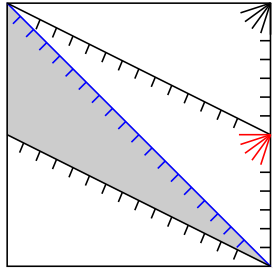
\[
\surfpicture{0}
\]
We take the constant sheaf on the shaded triangle, where the blue side is included, and other sides are not.
\end{proof}
\end{lemma}

By the same argument as in the conifold case: 

\begin{lemma}
\[\wSh_{\LambdaWsmall}(T)\simeq \la \wSh_{\Lambda_\flopSide}(T), \microsky_{x,\xi}\ra.\]
\end{lemma}

\begin{remark}
The skeleton $\Lambda_{+}$ is the usual FLTZ skeleton, as shown in Figure~\ref{figure surface skeleta}. Again, we have a microlocal skyscraper on a point in $\LambdaW\bs \Lambda_+$ as follows:
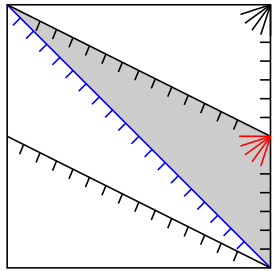
\[
\surfpicture{4}
\]
We can prove a semi-orthogonal decomposition for this case: \[\wSh_{\LambdaWsmall}(T)\simeq \la \wSh_{\Lambda_+}(T), \microsky_{x',\xi'}\ra.\]
\end{remark}

In the same way as in the conifold case, we can get a mirror equivalence between schobers, as follows.

\begin{theorem}\label{theorem mirror Fukaya surf} In the surface setting of Section~\ref{section surface geometry}, the analog of Theorem~\ref{theorem mirror Fukaya} holds.
\end{theorem}
\subsection{Proof for flober}
\label{section BKS schober equiv}

The coherent-constructible correspondence and an argument similar to the one presented in the last two subsections also gives a mirror equivalence of flobers. Let us only describe the result. 
We have an opposite flober
\[
 \begin{tikzpicture}[scale=2]
\node (Bminus) at (-1.3,1) {$ D(X_-) $};
\node (Bzero) at (0,1) {$ D(X_B) $};
\node (Bplus) at (+1.3,1) {$ D(X_+) $};
\draw[right hook->,transform canvas={yshift=+\arrowGap}] (Bminus) to node[above] {\scriptsize $\pr_-^!$} (Bzero);
\draw[left hook->,transform canvas={yshift=+\arrowGap}] (Bplus) to node[above] {\scriptsize $\pr_+^!$} (Bzero);
\draw[<-,transform canvas={yshift=-\arrowGap}] (Bminus) to node[below] {\scriptsize $\pr_{-*}$} (Bzero);
\draw[<-,transform canvas={yshift=-\arrowGap}] (Bplus) to node[below] {\scriptsize $\pr_{+*}$} (Bzero);
\end{tikzpicture}
\]
which is anti-equivalent to a flober as follows.
\[
 \begin{tikzpicture}[scale=2]
\node (Bminus) at (-1.3,1) {$ D(X_-) $};
\node (Bzero) at (0,1) {$ D(X_B) $};
\node (Bplus) at (+1.3,1) {$ D(X_+) $};
\draw[right hook->,transform canvas={yshift=+\arrowGap}] (Bminus) to node[above] {\scriptsize $\pr_-^*$} (Bzero);
\draw[left hook->,transform canvas={yshift=+\arrowGap}] (Bplus) to node[above] {\scriptsize $\pr_+^*$} (Bzero);
\draw[<-,transform canvas={yshift=-\arrowGap}] (Bminus) to node[below] {\scriptsize $\pr_{-*}$} (Bzero);
\draw[<-,transform canvas={yshift=-\arrowGap}] (Bplus) to node[below] {\scriptsize $\pr_{+*}$} (Bzero);
\end{tikzpicture}
\]
We denote the latter one by $\cP_B$.

On the other hand, the coherent-constructible correspondence $K$ takes the former opposite flober to another flober
\[
 \begin{tikzpicture}[scale=2]
\node (Bminus) at (-1.5,1) {$ \wSh_{\Lambda_-}(T) $};
\node (Bzero) at (0,1) {$ \wSh_{\Lambda_B}(T) $};
\node (Bplus) at (+1.5,1) {$ \wSh_{\Lambda_+}(T) $};
\draw[right hook->,transform canvas={yshift=+\arrowGap}] (Bminus) to node[above] {\scriptsize $\mpr_{-}^*$} (Bzero);
\draw[left hook->,transform canvas={yshift=+\arrowGap}] (Bplus) to node[above] {\scriptsize $\mpr_{+}^*$} (Bzero);
\draw[<-,transform canvas={yshift=-\arrowGap}] (Bminus) to node[below] {\scriptsize $\mpr_{-!}$} (Bzero);
\draw[<-,transform canvas={yshift=-\arrowGap}] (Bplus) to node[below] {\scriptsize $\mpr_{-!}$} (Bzero);
\end{tikzpicture}
\]
and then Theorem \ref{theorem GPS} takes this to an anti-equivalent flober as follows.
\[
 \begin{tikzpicture}[scale=2,xscale=1.5]
\node (Bminus) at (-1.5,1) {$ \cW_{\Lambda_-^\infty}(\cotangent{T}) $};
\node (Bzero) at (0,1) {$ \cW_{\Lambda_B^\infty}(\cotangent{T}) $};
\node (Bplus) at (+1.5,1) {$ \cW_{\Lambda_+^\infty}(\cotangent{T}) $};
\draw[right hook->,transform canvas={yshift=+\arrowGap}] (Bminus) to node[above] {\scriptsize $\mpr_{-}^*$} (Bzero);
\draw[left hook->,transform canvas={yshift=+\arrowGap}] (Bplus) to node[above] {\scriptsize $\mpr_{+}^*$} (Bzero);
\draw[<-,transform canvas={yshift=-\arrowGap}] (Bminus) to node[below] {\scriptsize $\mpr_{-*}$} (Bzero);
\draw[<-,transform canvas={yshift=-\arrowGap}] (Bplus) to node[below] {\scriptsize $\mpr_{-*}$} (Bzero);
\end{tikzpicture}
\]
We denote this by $\cP_A$.

\begin{theorem}\label{theorem flober equiv}
The two flobers $\cP_A$ and $\cP_B$ are equivalent via mirror symmetry.
\end{theorem}

\section{Applications to singularities}
\label{section sing HMS}

We first explain work of \BKandS{} calculating cohomology of the flober~$\cP_B$, before giving an analogous calculation for the flober~$\cP_A$, proving Proposition~\ref{keyproposition a-side cohom} and Corollary~\ref{keycorollary sing equiv}.

\subsection{B-side calculation}

Bondal--Kapranov--Schechtman showed the following:
\begin{proposition}{}\cite[Proposition 2.12]{BonKapSch}
The flober $\cP_B$ has the 2nd compact cohomology
\[
\bH^2_c(\Delta, \cP_B)\simeq D(X_0),
\]
meaning that the diagram
\begin{equation}\label{pushout}
\begin{tikzpicture}[xscale=1,scale=1.5,baseline={([yshift=-.5ex]current bounding box.center)}]

\node (Bminus) at (0,1) {$ D(X_B) $};
\node (Bplus) at (1,0) {$ D(X_+) $};

\node (Cminus) at (-1,0) {$ D(X_-) $};
\node (Cplus) at (0,-1) {$ D(X_0) $};

\draw[->] (Bminus) to node[above right] {\scriptsize $\pr_{+*}$}  (Bplus);
\draw[->] (Cminus) to node[below left] {\scriptsize $f_{-*}$}  (Cplus);

\draw[<-] (Cminus) to node[above left] {\scriptsize $\pr_{-*}$}  (Bminus);
\draw[<-] (Cplus) to node[below right] {\scriptsize $f_{+*}$}  (Bplus);

\end{tikzpicture}
\end{equation}
is a push-out in the Morita model category of dg-categories, where the $f_\pm$ are the resolutions $X_\pm \to X_0$. 
\end{proposition}

Let us recall their logic. First we extend the diagram to the following, defining $\cL_+$ (respectively~$\cC_-$) to be the kernel of $p_{+*}$ (respectively~$f_{-*}$).
\[
\begin{tikzpicture}[xscale=1,scale=1.5]

\node (Aminus) at (-1,2) {$ \cL_+ $};
\node (Aplus) at (-2,1) {$ \cC_- $};
\node (Aphantom) at (2,1) {$ \phantom{\cC_-} $};

\node (Bminus) at (0,1) {$ D(X_B) $};
\node (Bplus) at (1,0) {$ D(X_+) $};

\node (Cminus) at (-1,0) {$ D(X_-) $};
\node (Cplus) at (0,-1) {$ D(X_0) $};

\draw[->] (Aminus) to node[above left] {\scriptsize $h$}  (Aplus);
\draw[->] (Aminus) to (Bminus);
\draw[->] (Aplus) to (Cminus);

\draw[->] (Bminus) to node[above right] {\scriptsize $\pr_{+*}$}  (Bplus);
\draw[->] (Cminus) to node[below left] {\scriptsize $f_{-*}$}  (Cplus);

\draw[<-] (Cminus) to node[above left] {\scriptsize $\pr_{-*}$}  (Bminus);
\draw[<-] (Cplus) to node[below right] {\scriptsize $f_{+*}$}  (Bplus);

\end{tikzpicture}
\]

\begin{lemma}{}\cite[Lemma 2.14]{BonKapSch}\label{BKSlemma1}
In this situation, we have $D(X_B)/\cL_+\simeq D(X_+)$ and $D(X_-)/\cC_-\simeq D(X_0)$.
\end{lemma}

We moreover have the following general proposition, which is a slight modification of {\cite[Lemma 2.17]{BonKapSch}}.
\begin{proposition}
Let $u\colon \cS_1\rightarrow \cT_1$ be a fully-faithful dg-functor between Karoubian pre-triangulated dg-categories and $v\colon \cT_1\rightarrow \cT_2$ be a dg-functor. Let $\cS_2$ be the thick triangulated hull of objects of $v(\cS_1)$. Then the push-out of 
\[
 \begin{tikzpicture}[xscale=1,scale=1.5]

\node (Bminus) at (0,1) {$ \cT_1 $};
\node (Bzero) at (1,0) {$ \cT_1/\cS_1$};

\node (Aminus) at (-1,0) {$ \cT_2 $};

\draw[->] (Bminus) to node[above left] {\scriptsize $v$} (Aminus);

\draw[->] (Bminus) to (Bzero);

\end{tikzpicture}
\]
is Morita-equivalent to $\cT_2/\cS_2$. 
\end{proposition}
\begin{proof}
Let $\cS_2'$ be the full sub dg-category of $\cT_2$ spanned by objects of $v(\cS_1)$. Then \cite[Lemma 2.17]{BonKapSch} says the desired push-out is quasi-equivalent to $\cT_2/\cS_2'$. Since $\cT_2/\cS_2'$ is Morita-equivalent to $\cT_2/\cS_2$ (for example, one can deduce it from \cite[Lemma 2.16]{BonKapSch}), we complete the proof.
\end{proof}

The situation of the above proposition is: 
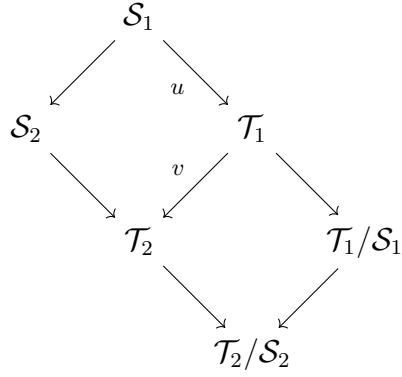
\[
 \begin{tikzpicture}[xscale=1,scale=1.5]

\node (Bminus) at (-1,2) {$ \cS_1 $};
\node (Bzero) at (0,1) {$ \cT_1 $};
\node (Bplus) at (1,0) {$ \cT_1/\cS_1 $};

\node (Aminus) at (-2,1) {$ \cS_2 $};
\node (Aphantom) at (2,1) {$ \phantom{\cS_2} $};
\node (Azero) at (-1,0) {$ \cT_2 $};
\node (Aplus) at (0,-1) {$ \cT_2/\cS_2 $};

\draw[->] (Bminus) to node[above,sloped] {} (Aminus);
\draw[->] (Bzero) to node[above left] {\scriptsize $v$}  (Azero);
\draw[->] (Bplus) to node[above,sloped] {} (Aplus);

\draw[->] (Bminus) to node[below left] {\scriptsize $u$}  (Bzero);
\draw[->] (Bzero) to node[above] {} (Bplus);

\draw[->] (Aminus) to node[above] {} (Azero);
\draw[->] (Azero) to node[above] {} (Aplus);

\end{tikzpicture}
\]
To apply this proposition to our situation, the following is sufficient.
\begin{lemma}{}\label{BKSLemma2}\cite[Lemma 2.18]{BonKapSch}
The functor $h\colon \cL_+\rightarrow \cC_-$ is a split generation. 
\end{lemma}
This completes the logic to prove $D^b(X_0)\simeq \text{push-out of (\ref{pushout})}$.

\subsection{A-side calculation}
Let us now turn to the A-side. We would like to prove the following:
\begin{proposition}\label{proposition a-side cohom}
The A-model flober $\cP_A$ has 2nd compact cohomology
\[
\bH^2_c(\Delta, \cP_A)\simeq \cW_{\bigcap\!\Lambda_\pm^\infty}(\Omega_T).
\]
\end{proposition}

We would first like to set up a somewhat general situation. Let $Z$ be a real analytic manifold. Let $\Lambda$ be a conic Lagrangian in $\Omega_Z$ and $\Lambda'$ be a closed conic Lagrangian subset of $\Lambda$. Then there exists a canonical inclusion \[\iota\colon \lSh_{\Lambda'}(Z)\rightarrow \lSh_{\Lambda}(Z)\] and a left adjoint $\iota^{\Ladj}\colon \lSh_{\Lambda}(Z)\rightarrow \lSh_{\Lambda'}(Z)$.

\begin{notation}Let $S_{\Lambda}(\Lambda\bs \Lambda')$ be the set of microlocal skyscraper sheaves in $\lSh_{\Lambda}(Z)$ over $\Lambda\bs \Lambda'$ and $S_{\Lambda}(\Lambda')$ be the set of microlocal skyscraper sheaves in $\lSh_{\Lambda}(Z)$ over $\Lambda'$. 
\end{notation}

For brevity, we set \[S:=S_{\Lambda}(\Lambda\bs \Lambda')\cup S_{\Lambda}(\Lambda') \qquad \text{and} \qquad R:=S_{\Lambda}(\Lambda\bs \Lambda').\]
Then $\iota^{\Ladj}$ maps $R$ to $0$ and $S\bs R$ to a set of generators of $\lSh_{\Lambda'}(Z)$. Then $\lSh_{\Lambda}(Z)=\la S\ra$ by Nadler's generation result~\cite{N}. By Thomason's localization theorem {\cite[Theorem 1.14]{Ne}}, we have
\begin{equation*}
\la R\ra^{\idem} = \la R\ra \cap \wSh_{\Lambda}(Z)
\end{equation*}
where $\la R\ra^{\idem}$ is the smallest thick subcategory containing $R$ and 
\begin{equation*}
\big(\lSh_{\Lambda}(Z)/\!\la R\ra\big)^c\simeq \wSh_{\Lambda}(Z)/\!\la R\ra\cap \wSh_{\Lambda}(Z).
\end{equation*}
where $(-)^c$ is the full subcategory spanned by compact objects (more precisely $\aleph_0$-compact). Combining these, we have  the following.
\begin{equation*}
\big(\lSh_{\Lambda}(Z)/\!\la R\ra\big)^c\simeq \wSh_{\Lambda}(Z)/\!\la R\ra^{\idem}
\end{equation*}

To prove the next lemma, we recall some facts. 
\begin{definition}[Bousfield localization]
Let $\cT$ be a triangulated category and $\cS$ be a thick subcategory. The Bousfield localization functor for the pair $(\cT, \cS)$ is a right adjoint of the quotient functor $\cT\rightarrow \cT/\cS$.
\end{definition}

\begin{theorem}{}\cite[Theorem 9.1.16]{Ne}
The Bousfield localization functor is fully faithful.
\end{theorem}

Now we would like to prove the following:
\begin{lemma}
\[
\lSh_{\Lambda}(Z)/\!\la R\ra \simeq \lSh_{\Lambda'}(Z).
\]
\end{lemma}
\begin{proof}
Again, general nonsense tells us that there exists a right adjoint $\catquot_1^{\Radj}$ of the quotient $\catquot_1\colon \lSh_\Lambda(Z)\rightarrow \lSh_\Lambda(Z)/\!\la R\ra$ (cf. \cite[Example 8.4.5]{Ne}). Hence this is a Bousfield localization and hence fully faithful. For an object $\cE\in  \lSh_\Lambda(Z)/\!\la R\ra$ and an $r\in R$, we have 
\[
\begin{split}
\Hom_{\lSh_\Lambda(Z)}(r, \catquot_1^{\Radj}(\cE))&\simeq \Hom_{\lSh_\Lambda(Z)/\la R\ra}(\catquot_1(r), \cE)\simeq 0.
\end{split}
\]
Hence $\SS(\catquot_1^{\Radj}(\cE))\subset \Lambda'$ by the definition of microlocal skyscraper sheaves. So we have a fully faithful functor $\widetilde{\catquot_1^{\Radj}}\colon \lSh_\Lambda(Z)/\!\la R\ra\rightarrow \lSh_{\Lambda'}(Z)$ such that $\catquot_1^{\Radj}=\iota\circ \widetilde{\catquot_1^{\Radj}}$. We would like to see that this is essentially surjective. 

Let us take $\cF\in\lSh_{\Lambda'}(Z)$. It is enough to prove $\widetilde{\catquot_1^{\Radj}}\circ \catquot_1\circ \iota (\cF)\simeq \cF$ for our purpose. Letting $\catquot_2$ be the functor $\lSh_{\Lambda}(Z)/\!\la R\ra\rightarrow \lSh_{\Lambda'}(Z)$ induced by~$\iota^{\Ladj}$, which satisfies $\catquot_2\circ \catquot_1= \iota^{\Ladj}$.

First, let us see $\pi_2$ is the left adjoint of $\pi_1\circ \iota$. Let us consider an object $\pi_1(\cG)\in \lSh_\Lambda(Z)/ \la R\ra$ which is represented by $\cG\in \lSh_\Lambda(Z)$. We have
\begin{equation*}
\begin{split}
\Hom_{\lSh_{\Lambda'}(Z)}(\pi_2\circ\pi_1(\cG), \cF) &\simeq \Hom_{\lSh_{\Lambda'}(Z)}(\iota^{\Ladj}(\cG),\cF)\\
&\simeq \Hom_{\lSh_{\Lambda}(Z)}(\cG, \iota(\cF))\\
&\simeq \Hom_{\lSh_{\Lambda}(Z)/\la R\ra}(\pi_1(\cG), \pi_1\circ \iota(\cF)).
\end{split}
\end{equation*}
For the last equality, we used the fact that $\la R\ra$ is in the left orthogonal of~$\lSh_{\Lambda'}(Z)$. Next, we can see $\pi_1\circ \iota$ is fully faithful because of the following.
\begin{equation*}
\begin{split}
\Hom_{\lSh_{\Lambda}(Z)/\la R\ra}(\pi_1\circ \iota(\cF'), \pi_1\circ \iota(\cF))&\simeq \Hom_{\lSh_{\Lambda'}(Z)}(\pi_2\circ \pi_1\circ \iota(\cF'), \cF)\\
&\simeq \Hom_{\lSh_{\Lambda'}(Z)}(\iota^{\Ladj}\circ \iota(\cF'), \cF)\\
&\simeq \Hom_{\lSh_{\Lambda'}(Z)}(\cF', \cF)
\end{split}
\end{equation*}
Finally, we have
\begin{equation*}
\begin{split}
\Hom_{\lSh_{\Lambda'}(Z)}\big(\cF', \widetilde{\catquot_1^{\Radj}}\circ \catquot_1\circ \iota (\cF)\big)&\simeq \Hom_{\lSh_{\Lambda}(Z)}(\iota(\cF'), \catquot_1^{\Radj}\circ \catquot_1\circ \iota(\cF))\\
&\simeq \Hom_{\lSh_\Lambda(Z)/\la R\ra}(\catquot_1\circ \iota (\cF'), \catquot_1\circ \iota(\cF))\\
&\simeq \Hom_{\lSh_{\Lambda'}(Z)}(\cF',\cF).
\end{split}
\end{equation*}
Yoneda then completes the proof.
\end{proof}

\begin{lemma}
The functor $\catquot_2$ is an equivalence.
\end{lemma}
\begin{proof}
From the proof above, we also see that $\catquot_1\circ \iota \colon \lSh_{\Lambda'}(Z)\rightarrow \lSh_{\Lambda}(Z)/\!\la R\ra$ is an equivalence. Since $\pi_2$ is the left adjoint of $\pi_1\circ \iota$, we actually have \[\pi_2=\widetilde{\pi_1^{\Radj}}.\qedhere\]
\end{proof}
\begin{corollary}
The functor $\catquot_2$ induces an equivalence as follows.
\begin{equation*}
\wSh_{\Lambda'}(Z)\simeq \wSh_{\Lambda}(Z)/\!\la R\ra^{\idem}
\end{equation*}
\end{corollary}

Let us go back to our situation. We now have a diagram.
\[
\begin{tikzpicture}[xscale=1,scale=1.75]

\node (Aminus) at (-1,2) {$ \big\langle S_{\Lambda_B}(\Lambda_B\bs \Lambda_+)\big\rangle^{\idem} $};
\node (Aplus) at (-2,1) {$ \big\langle S_{\Lambda_-}(\Lambda_-\bs \bigcap\!\Lambda_\pm)\big\rangle^{\idem} $};
\node (Aphantom) at (2,1) {$ \phantom{\big\langle \Lambda_-\bs (\bigcap\!\Lambda_\pm)\big\rangle^{\idem}} $};

\node (Bminus) at (0,1) {$ \wSh_{\Lambda_B}(T) $};
\node (Bplus) at (1,0) {$\wSh_{\Lambda_+}(T) $};

\node (Cminus) at (-1,0) {$ \wSh_{\Lambda_-}(T) $};
\node (Cplus) at (0,-1) {$ \wSh_{\bigcap\!\Lambda_\pm}(T) $};

\draw[->] (Aminus) to  node[above left] {\scriptsize $\reflectbox{\it{h}}$} (Aplus);
\draw[->] (Aminus) to (Bminus);
\draw[->] (Aplus) to (Cminus);

\draw[->] (Bminus) to node[above right] {\scriptsize $\mpr_{+!}$}  (Bplus);
\draw[->] (Cminus) to node[pos=\labelAdjust,below left,transform canvas={xshift=8pt}] {\scriptsize $\reflectbox{\it{f}}_{-!}$} (Cplus);

\draw[<-] (Cminus) to node[above left] {\scriptsize $\mpr_{-!}$}  (Bminus);
\draw[<-] (Cplus) to node[pos=1-\labelAdjust,below right] {\scriptsize $\reflectbox{\it{f}}_{+!}$} (Bplus);

\end{tikzpicture}
\]
The two sequences from upper left to lower right are Verdier--Drinfeld quotients by Corollary 6.5. The functor $\reflectbox{\it{h}}$ is the restriction of $\mpr_{-!}$ to the uppermost category. The functors $\,\reflectbox{\it{f}}_{\pm!}$ are the restrictions of the left adjoints of the inclusions $\wSh_{\bigcap\!\Lambda_\pm}(T)\subset \wSh_{\Lambda_\pm}(T)$.

\begin{lemma}
The category $\big\langle \Lambda_-\bs \bigcap\!\Lambda_\pm \big\rangle^{\idem}$ is split-generated by the image of $\reflectbox{\it{h}}$.
\end{lemma}
\begin{proof}
Note that $\mpr_{-!}$ takes a microlocal skyscraper sheaf in $\wSh_{\Lambda_B}(T)$ over a point in $\Lambda_-$ to a microlocal skyscraper sheaf in $\wSh_{\Lambda_-}(T)$ over the same point in $\Lambda_-$~\cite{N}. Note also that $\mpr_{-!}$ takes a microlocal skyscraper sheaf in $\wSh_{\Lambda_B}(T)$ over a point in $\Lambda_B\bs \Lambda_{-}$ to zero~\cite{N}. These imply the well-definedness of $\reflectbox{\it{h}}$. Since $\Lambda_B\bs \Lambda_+\supset \Lambda_-\bs \bigcap\!\Lambda_\pm$, these also imply the surjectivity of $\reflectbox{\it{h}}$ on the split generators.
\end{proof}

Repeating the logic presented in the beginning of the section, we can conclude
\[
\begin{tikzpicture}[xscale=1,scale=1.5]

\node (Bminus) at (0,1) {$ \wSh_{\Lambda_B}(T) $};
\node (Bplus) at (1,0) {$\wSh_{\Lambda_+}(T) $};

\node (Cminus) at (-1,0) {$ \wSh_{\Lambda_-}(T) $};
\node (Cplus) at (0,-1) {$ \wSh_{\bigcap\!\Lambda_\pm}(T) $};

\draw[->] (Bminus) to node[above right] {\scriptsize $\mpr_{+!}$}  (Bplus);
\draw[->] (Cminus) to node[pos=\labelAdjust,below left,transform canvas={xshift=8pt}] {\scriptsize $\reflectbox{\it{f}}_{-!}$} (Cplus);

\draw[<-] (Cminus) to node[above left] {\scriptsize $\mpr_{-!}$}  (Bminus);
\draw[<-] (Cplus) to node[pos=1-\labelAdjust,below right] {\scriptsize $\reflectbox{\it{f}}_{+!}$} (Bplus);

\end{tikzpicture}
\]
is a homotopy push-out in the Morita model. Noting that $(\bigcap\!\Lambda_\pm)^\infty=\bigcap\!\Lambda_\pm^\infty$ by the definition of taking infinity, Theorem \ref{theorem GPS} takes this diagram to another push-out diagram.
\[
\begin{tikzpicture}[xscale=1,scale=1.5]

\node (Bminus) at (0,1) {$ \cW_{\Lambda_B^\infty}(\Omega_T) $};
\node (Bplus) at (1,0) {$\cW_{\Lambda^\infty_+}(\Omega_T) $};

\node (Cminus) at (-1,0) {$ \cW_{\Lambda^\infty_-}(\Omega_T) $};
\node (Cplus) at (0,-1) {$ \cW_{\bigcap\!\Lambda_\pm^\infty}(\Omega_T) $};

\draw[->] (Bminus) to node[above right] {\scriptsize $\mpr_{+*}$}  (Bplus);
\draw[->] (Cminus) to node[pos=\labelAdjust,below left,transform canvas={xshift=8pt}] {\scriptsize $\reflectbox{\it{f}}_{-*}$} (Cplus);

\draw[<-] (Cminus) to node[above left] {\scriptsize $\mpr_{-*}$}  (Bminus);
\draw[<-] (Cplus) to node[pos=1-\labelAdjust,below right] {\scriptsize $\reflectbox{\it{f}}_{+*}$} (Bplus);

\end{tikzpicture}
\]
Here $\,\reflectbox{\it{f}}_{\pm*}$ is defined as the composition of the functor in Theorem~\ref{theorem GPS} and~$\,\reflectbox{\it{f}}_{\pm!}$. This implies Proposition~\ref{proposition a-side cohom}.

On the other hand, Theorem~\ref{theorem flober equiv} implies that $\bH^2_c(\Delta, \cP_A)\simeq \bH^2_c(\Delta, \cP_B)$. Combining with Proposition~\ref{proposition a-side cohom}, we have homological mirror symmetry for singular varieties
\begin{corollary}\label{cor sing equiv}
\[
D(X_0)\simeq \cW_{\bigcap\!\Lambda_\pm^\infty}(\Omega_T).
\]
\end{corollary}

\begin{remark}
In this section, we used flobers to study the compact cohomology rather than spherical pairs. If one instead uses the spherical pairs from Theorem~\ref{theorem schober equivalence}, then one will arrive at the same conclusion as presented here, i.e.~the flobers and spherical pairs have the same 2nd~compact cohomology.
\end{remark}

\begin{remark}
The method presented here might be generalized to another general proof of the coherent-constructible correspondence for singular toric varieties, as proved by the second author in~\cite{K2}.

Let $X$ be a singular toric variety. Let $S_X$ be the category consisting of 
\begin{enumerate}
\item Object: smooth toric Deligne--Mumford stack refining $X$.
\item Morphism: morphism along refinement.
\end{enumerate}
There exists a functor $\cS_X$ from $S_X$ to the category of dg-categories defined by $X'\mapsto D(X')$ and morphisms are mapped to push-forwards along them. 
\begin{conjecture}
The universality morphism gives an equivalence
\begin{equation*}
\lim_{\substack{\longrightarrow\\ X'\in \cS_X}} D(X')\simeq D(X).
\end{equation*}
\end{conjecture}
One can define the A-model counterpart and conjecture
\begin{conjecture}
The universality morphism gives an equivalence
\begin{equation*}
\lim_{\substack{\longrightarrow\\ X'\in \cS_X}} \wSh_{\Lambda_{X'}}(T)\simeq \wSh_{\Lambda_X}(T).
\end{equation*}
\end{conjecture}
These two conjectures would allow us to conclude an equivalence between $D(X)\simeq \wSh_{\Lambda_X}(T)$ from the smooth coherent-constructible correspondence.
\end{remark}


\end{document}